\documentclass[11pt]{amsart}
\usepackage{fullpage}

\usepackage{amsmath}
\usepackage{amsfonts}
\usepackage{amssymb}
\usepackage{amsthm}
\usepackage{mathtools} % An improvement of amsmath
\usepackage{latexsym}

\usepackage{physics}

\usepackage{enumerate}
\usepackage{cancel}
\usepackage{cases}
\usepackage{empheq}
\usepackage{multicol}

\usepackage{graphics} %% Conflicts with pdflatex.

\usepackage{wrapfig}

\usepackage{color}

\usepackage[backgroundcolor=gray!30,linecolor=black]{todonotes}
\usepackage{mathrsfs}
\usepackage{fontenc} %T1 font encoding
\usepackage{inputenc} %UTF-8 support

\usepackage{verbatim}

\theoremstyle{plain}
\newtheorem{theorem}{Theorem}[section]

\newtheorem{lemma}[theorem]{Lemma}

\theoremstyle{definition}

\theoremstyle{remark}
\newtheorem{remark}[theorem]{Remark}

\numberwithin{equation}{section} %% Equation numbering control.
\numberwithin{figure}{section}   %% Figure numbering control.

\usepackage[square,comma,numbers,sort&compress]{natbib}
\usepackage[colorlinks=true, pdfborder={ 00 0}]{hyperref}
\hypersetup{urlcolor=blue, citecolor=red}
\usepackage{url}

\newcommand{\field}[1]{\mathbb{#1}}

\newcommand{\nR}{\field{R}}

\newcommand{\pd}[2]{\frac{\partial #1}{\partial #2}}
\renewcommand{\abs}[1]{\left\lvert#1\right\rvert}
\newcommand{\set}[1]{\left\{#1\right\}}
\renewcommand{\ip}[2]{\left<#1,#2\right>}

\renewcommand{\norm}[1]{\left\|#1\right\|}

\newcommand{\normL}[1]{\|#1\|_{L^2}}
\newcommand{\normH}[1]{\|#1\|_{H^1}}
\newcommand{\fournorm}[1]{\|#1\|_{L^4}}

\newcounter{my_counter}
\renewcommand{\grad}{\nabla}
\renewcommand{\div}{\nabla \cdot} 
\newcommand{\lap}{\Delta} 

\usepackage{placeins} % Used for \FloatBarrier

\usepackage{cleveref}

\renewcommand{\braket}[1]{\left<#1\right>}

\title[Sparse-in-time data assimilation]{%Living in the now:
The second-best way to do\\sparse-in-time continuous data assimilation:\\ Improving convergence rates\\for the 2D and 3D Navier-Stokes Equations
}

\date{\today}

\author{Adam Larios}
\address[Adam Larios]{Department of Mathematics, 
                University of Nebraska--Lincoln,
        Lincoln, NE 68588-0130, USA}
\email[Adam Larios]{alarios@unl.edu}
\author{Yuan Pei}
\address[Yuan Pei]{Department of Mathematics, 
                Western Washington University,
        Bellingham, WA 98225-9063, USA}
\email[Yuan Pei]{peiy@wwu.edu}
\author{Collin Victor}
\address[Collin Victor]{Department of Mathematics, 
                University of Nebraska--Lincoln,
        Lincoln, NE 68588-0130, USA}
\email[Collin Victor]{collin.victor@huskers.unl.edu}
\begin{document}

%==============================================================-
\begin{abstract}
We study different approaches to implementing sparse-in-time observations into the the Azouani-Olson-Titi data assimilation algorithm.  We propose a new method which introduces a ``data assimilation window'' separate from the observational time interval.  We show that by making this window as small as possible, we can drastically increase the strength of the nudging parameter without losing stability. Previous methods used old data to nudge the solution until a new observation was made.  In contrast, our method stops nudging the system almost immediately after an observation is made, allowing the system relax to the correct physics.  We show that this leads to an order-of-magnitude improvement in the time to convergence in our 3D Navier-Stokes simulations.  Moreover, our simulations indicate that our approach converges at nearly the same rate as the idealized method of direct replacement of low Fourier modes proposed by Hayden, Olson, and Titi (HOT). However, our approach can be readily adapted to non-idealized settings, such as finite element methods, finite difference methods, etc., since there is no need to access Fourier modes as our method works for general interpolants. It is in this sense that we think of our approach as ``second best;'' that is, the ``best'' method would be the direct replacement of Fourier modes as in HOT, but this idealized approach is typically not feasible in physically realistic settings.  While our method has a convergence rate that is slightly sub-optimal compared to the idealized method, it is directly compatible with real-world applications.

% We also introduce a new algorithm allowing for linear extrapolation in time of the observed data.  
Moreover, we prove analytically that these new algorithms are globally well-posed, and converge to the true solution exponentially fast in time.  
In addition, we provide the first 3D computational validation of HOT algorithm.%  words,  characters
\end{abstract}

% REQUIRED

% REQUIRED
 % PDEs in connection with control and optimization 

\maketitle
\thispagestyle{empty}%Gets rid of page number on first page.
%============================================================

\noindent

\section{Introduction}

\noindent
If one wishes to incorporate infrequent observational data into a physical model (e.g., a model of the weather), an immediate question arises: \textit{What should one do to the model during the times when there is no data?}  If one is using some kind of data assimilation scheme, a natural answer is that the correction from the data assimilation should use the most recent data to continue to steer the simulation toward the observed data.  Indeed, this is precisely what was proposed and studied in \cite{Foias_Mondaini_Titi_2016}.  However, in the present work, we show that another approach appears to work even better.  Roughly speaking, our answer to the above question is: \textit{Do nothing}.  More precisely, we show computationally in 3D simulations of the Navier-Stokes equations using Azouani-Olson-Titi (AOT)-style data assimilation that a very brief but very hard nudging of the solution toward the most recent interpolated data, and then a subsequent period of letting the solution relax toward the correct physics, can lead to an order-of-magnitude improvement in the time to convergence.  

% Note: Here is ChatGPT's enhancement of the above paragraph:
% =====
% When dealing with physical models that incorporate infrequent observational data, a common challenge is how to handle periods when there is no available data. Traditionally, data assimilation schemes have relied on using the most recent data to steer the simulation towards observed values during these data gaps. While this approach has been studied and implemented in previous works, we have discovered that an alternative approach may yield even better results. In our research, we propose a new strategy that involves allowing the solution to do nothing during data gaps, followed by a brief but intense nudging towards the most recent interpolated data. After this initial nudging period, we observe that allowing the solution to relax towards the correct physical behavior can lead to an order-of-magnitude improvement in the time to convergence. To demonstrate the effectiveness of this approach, we apply it to 3D simulations of the Navier-Stokes equations using Azouani-Olson-Titi style data assimilation, and show that it outperforms previous approaches in terms of convergence speed and accuracy. This new approach represents an exciting development in the field of data assimilation, and has the potential to significantly improve the accuracy and efficiency of physical models that rely on infrequent observational data.

One problem that arises in the modeling of dynamical systems of mathematical physics is that of initialization. That is, if one could evolve the system forward in time perfectly, one would still require an accurate initial state to evolve forward from. \emph{Data assimilation} is a class of techniques that mitigates this problem by incorporating data from the dynamical system with the underlying physical model in order to more accurately simulate the dynamics. A new approach to data assimilation has emerged recently, known as the Azouani-Olson-Titi (AOT) algorithm \cite{Azouani_Titi_2014,Azouani_Olson_Titi_2014}. Since its inception, the AOT algorithm has been the subject of much recent study in both analytical studies  
\cite{
Albanez_Nussenzveig_Lopes_Titi_2016,
Bessaih_Olson_Titi_2015,
Biswas_Bradshaw_Jolly_2020,
Biswas_Foias_Mondaini_Titi_2018downscaling,
Biswas_Martinez_2017,
Biswas_Price_2020_AOT3D,
Carlson_Hudson_Larios_2018,
Carlson_Larios_2021_sens,
Chen_Li_Lunasin_2021,
Celik_Olson_2023,
Diegel_Rebholz_2021,
Du_Shiue2021,
Farhat_Jolly_Titi_2015,
Farhat_Lunasin_Titi_2016abridged,
Farhat_Lunasin_Titi_2016benard,
Farhat_Lunasin_Titi_2016_Charney,
Farhat_Lunasin_Titi_2017_Horizontal,
Foias_Mondaini_Titi_2016,
Foyash_Dzholli_Kravchenko_Titi_2014,
GarciaArchilla_Novo_Titi_2018,
GarciaArchilla_Novo_2020,
Gardner_Larios_Rebholz_Vargun_Zerfas_2020_VVDA,
GlattHoltz_Kukavica_Vicol_2014,
Ibdah_Mondaini_Titi_2018uniform,
Jolly_Martinez_Olson_Titi_2018_blurred_SQG,
Jolly_Martinez_Titi_2017,
Jolly_Sadigov_Titi_2015,
Larios_Pei_2018_NSV_DA,
Larios_Victor_2021_chiVsdelta2D,
Markowich_Titi_Trabelsi_2016_Darcy,
Mondaini_Titi_2018_SIAM_NA,
Pachev_Whitehead_McQuarrie_2021concurrent,
Pei_2019,
Rebholz_Zerfas_2018_alg_nudge,
Zerfas_Rebholz_Schneier_Iliescu_2019} 
and computational studies 
\cite{
Altaf_Titi_Knio_Zhao_Mc_Cabe_Hoteit_2015,
Carlson_VanRoekel_Petersen_Godinez_Larios_2021,
DiLeoni_Clark_Mazzino_Biferale_2018_inferring,
Desamsetti_Dasari_Langodan_Knio_Hoteit_Titi_2019_WRF,
Franz_Larios_Victor_2021,
Gesho_Olson_Titi_2015,
Larios_Rebholz_Zerfas_2018,
Larios_Victor_2019,
Lunasin_Titi_2015,
Larios_Pei_2017_KSE_DA_NL}.

In this work we examine a modification to the AOT algorithm using discrete in time observational data, similar to \cite{Foias_Mondaini_Titi_2016}. We note that the feedback-control term introduced in \cite{Foias_Mondaini_Titi_2016} is constant in between observations, and yet the feedback is applied until new observational data is generated. In practice this means that the feedback term is forcing the solution towards an old state long after that state has passed. For highly turbulent flows, after observations are gathered, the observational data quickly becomes outdated, so the nudging term becomes increasingly out-of-sync with the present state of the flow. Instead, we propose a modification to this scheme that introduces a new parameter $\tau$, denoting the length of the assimilation window, a time interval within which observational data is assimilated into the model equation.  When this assimilation window has passed, the observational data is discarded and the system is allowed to evolve without outside feedback control until new observations are made.

We show analytically that our proposed data assimilation scheme is well-posed and exhibits exponential convergence to the true solution, following similar methods to \cite{Foias_Mondaini_Titi_2016}. Moreover, we prove that the rate of convergence is comparable to that of the algorithm given in \cite{Foias_Mondaini_Titi_2016} for certain choices of parameters. Additionally, we extend our analysis to consider the case of certain explicit time-extrapolants to show global well-posedness and exponential convergence to the true solution.
We find computationally, that our algorithm allows for greater flexibility with the parameters. Utilizing larger values of the feedback-control parameter $\mu$ we obtain an order-of-magnitude improvement on the convergence time to the true solution when our algorithm is applied to the 3D incompressible Navier-Stokes equations.
% We also introduce a new algorithm allowing for linear extrapolation in time of the observed data.

The paper is organized as follows. In Section 2 we outline the functional setting of the Navier-Stokes equations, the formulation of the data assimilation equations, as well as standard results we will use in our analysis. In Section 3 we prove analytical results regarding both the well-posedness and the convergence for our assimilated equations. In Section 4 we discuss computational results for this algorithm when applied to the 3D Navier-Stokes equations. 

\begin{remark}
    During the final preparations of this manuscript, we became aware of the recent work \cite{Hammoud_LeMaitre_Titi_Hoteit_Knio_2023_DDS}, which proposes and computationally investigates an approach that the authors call Discrete Data Assimilation (DDA), which is quite similar to the approach presented here, in the context of the 2D Rayleigh-B\'enard convection system.  The present work differs from \cite{Hammoud_LeMaitre_Titi_Hoteit_Knio_2023_DDS} in that we provide analysis proving that our scheme is globally well-posed and converges, we provide analytical bounds for the parameters, 
    % we propose and study a linear-in-time extrapolation scheme, 
    and finally our simulations are in the context of the 3D Navier-Stokes equations.
\end{remark}

\section{Preliminaries}
In this article, we study the following incompressible Navier-Stokes equations in two spatial dimensions:
\begin{subequations}
\begin{empheq}[left=\empheqlbrace]{align}  
  \label{NSE}
    u_t + u\cdot \grad u - \nu \lap u &= \grad p + f,\\
    \div u &= 0.
\end{empheq}
\end{subequations}
Here $u(x,t) = (u_1(x_1, x_2,t),u_2(x_1, x_2,t))$ is the velocity, $\nu>0$ is the viscosity, $p$ is the pressure and $f$ is a forcing term. Applying the continuous data assimilation algorithm to this equation yields the following equations:
\begin{subequations}
\begin{empheq}[left=\empheqlbrace]{align}  
  \label{NSE-AOT}
    v_t + v\cdot \grad v - \nu \lap v &= \grad p + f + \mu I_{h,\kappa}(u-v),\\
    \div v &= 0.
\end{empheq}
\end{subequations}

Here $I_{h,\kappa}$ is a linear interpolation term with associated spatial length scale $h$ and time scale $\kappa$.
% We write $I_{h,\kappa}$ 
% We note that all of our analysis will work for any interpolant satisfyin
% In previous works, such as \cite{Foias_Mondaini_Titi_2016}, the linear interpolation of $I_{h,\kappa}$ was chosen to be piecewise-constant in time. 
% We note that out analysis 
% In this work we consider the more general class of explicit first-order interpolants in time.
% For observational data that is sparse in time we require an interpolation term in time as well as space.
Inspired by \cite{Foias_Mondaini_Titi_2016}, in this paper, we propose the following data assimilation algorithm:
\begin{subequations}
\begin{empheq}[left=\empheqlbrace]{align}  
  \label{NSE-v}
    v_t + v\cdot \grad v - \nu \lap v &= \grad p + f + \sum_{n=0}^\infty \mu I_{h,\kappa}[u-v]\chi_{n,\tau},\\
    \div v &= 0.
\end{empheq}
\end{subequations}
Here $\chi_{n,\tau}$ is the indicator function of the time interval $[t_n, t_n + \tau)$, where $t_{n+1} - t_n = \kappa \geq \tau$. Setting $\kappa=\tau$, one recovers the algorithm in \cite{Foias_Mondaini_Titi_2016}, while here we allow $\tau < \kappa$.

We require that the above interpolator $I_{h,\kappa}$ satisfies that following condition:
\begin{align}
  \label{interpolator_time}
      \norm{\xi(t) - I_{h,\kappa}\xi(t)}^2_{L^2} &\leq  c_0h^2\norm{\xi(t)}_{H^1}^2 + c_1(t-t_n) \int_{t_n}^t\norm{\dv{\xi(s)}{s}}_{L^2}^2 ds, \quad \text{ for all } \xi \in H^1.
\end{align}
Above we have $\kappa = t_{n+1} - t_{n}$, so we can use $t-t_n \leq \kappa$ for an upper bound. 
However we typically use $t-t_n \leq \tau$ for restricted values of $t$ as an upper bound in \ref{interpolator_time} as the indicator function in \cref{NSE-v} will remove the interpolant term for larger $t$ values.
Here and throughout this work we make reference to $t_n$, which denotes the times in which observations of \cref{NSE} are made with frequency $\kappa$ with the additional note that $t_0$ coincides that given in \Cref{thm1}.
We note that above we have opted to use any interpolant satisfying \cref{interpolator_time}, which corresponds to interpolants with a zero order explicit time-extrapolation e.g. piecewise-constant  time-extrapolation.
One can also consider higher order time-extrapolants that satisfy a similar inequality with a higher order time-derivative component.

\begin{remark}
    Note that in some real-world applications, the time between observations, $\kappa>0$, may not be fixed.  For the sake of simplicity, we assume that $\kappa$ is fixed in the present work.  However, our results could be readily adapted to allow for time-varying $\kappa$ because we only use the more general bound \cref{interpolator_time} rather than assuming, e.g., a fixed piece-wise constant time-extrapolation.  In principle, one could allow $I_{h,\kappa}$ to depend on $t_n$.
\end{remark}

\subsection{Functional Setting of Navier-Stokes Equations}
Throughout this paper, we denote the Lebesgue space and the Sobolev space by $L^p$ for $0\leq p \leq \infty$ and $H^s = W^{s,2}$ with $s>0$, respectively. Let $\mathcal V$ be the set of all $L$-periodic trigonometric polynomials form $\nR^2$ to $\nR^2$ that are divergence free with zero average. We denote by $H$, $V$, and $V'$ the closures of $\mathcal V$ in the $L^2(\Omega)^2$ and $H^1(\Omega)^2$ norms, and the dual space of $V$, respectively, with inner products on $H$ and $V$ as
\begin{align} (u, v) = \sum_{i=1}^2\int_{\mathbb{T}^2}u_{i}v_{i}\,dx \text{ \,\,and\,\,  } (\nabla u, \nabla v) = \sum_{i, j=1}^2\int_{\mathbb{T}^2}\partial_{j}u_{i}\partial_{j}v_{i}\,dx, \end{align}
respectively, associated with the norms $\| u \|_{H}=(u, u)^{1/2}$ and $\| u \|_{V}=(\nabla u, \nabla u)^{1/2}$, where $u(x)= u(x_1, x_2) = (u_1(x_1, x_2), u_2(x_1, x_2))$ and $v(x) = v(x_1, x_2)=(v_1(x_1, x_2), v_2(x_1, x_2))$. 
For the sake of convenience, we use $\normL{u}$ and $\normH{u}$ to denote the above norms in $H$ and $V$, respectively.

Applying the orthogonal Leray-Helmholtz projector, $P_\sigma$, to \cref{NSE,NSE-v} one formally obtains the following equations for $u$ and $v$, respectively:
\begin{align}
    \label{NSE-functional}
    u_t  + \nu Au +  B(u,u)  &= f,\\
\label{NSE-AOT-functional}
    v_t  + \nu Av +  B(v,v)  &= f + \sum_{n=0}^\infty \mu P_\sigma I_{h,\kappa}[u-v]\chi_{n,\tau}.
\end{align}
Here $A:=-P_\sigma\Delta$ and $B(u,v) := P_\sigma((u\cdot\nabla)v)$ which will be discussed in depth below.

The Stokes operator, $A$, has domain $\mathcal{D}(A):= H^2\cap V$ and can be extended to a linear operator from $V$ to $V'$ as
\begin{align} \left<Au, v\right> = (\nabla u, \nabla v) \text{  for all  } v\in V. \end{align}
Notice that on our domain $\Omega = \mathbb{T}^2$, $A = -\Delta P_{\sigma}$.
It is well-known that 
% $A^{-1} : H \hookrightarrow \mathcal{D}(A)$ 
% is a positive-definite, self-adjoint, and compact operator from $H$ into itself, which possesses an orthonormal basis of positive eigenfunctions $\{ w_{k}\}_{k=1}^{\infty}$ in $H$, corresponding to a sequence of non-increasing sequence of eigenvalues. 
% Therefore, 
$A$ has non-decreasing eigenvalues $\lambda_{k}$.
% , i.e., $0 \leq \lambda_1 \leq \lambda_2, \dots$. 
Due to this fact, the following Poincar\'e inqualities are valid:
\begin{align}\label{Poincare}
    \lambda_1\normL{u}^2
    \leq
    \normH{u}^2 \text{\,\, for\,\,} u\in V;\quad
    \lambda_1\normH{u}^2
    \leq
    \normL{Au}^2 \text{\,\, for\,\,} u\in D(A).
\end{align}
Thus, $\normL{u}$ is equivalent to $\normH{u}$. For a more in-depth look at the various properties of $A$ see e.g. \cite{Robinson_2001, Temam_1997_IDDS, Constantin_Foias_1988}.

In this paper we frequently utilize the following Ladyzhenskaya inequality 
\begin{align}\label{Ladyzhenskaya}
\fournorm{u}^2 \leq c\normL{u} \normH{u},
\quad\text{for all}\,\, u\in V,
\end{align}
which is a variation of the following interpolation result (see, e.g., \cite{Nirenberg_1959_AnnPisa} for a detailed proof).
Assume $1 \leq q, r \leq \infty$, and $0<\gamma<1$.  
For $v\in L^q(\mathbb{T}^{n})$, such that  $\partial^\alpha v\in L^{r} (\mathbb{T}^{n})$, for $|\alpha|=m$, then 
\begin{align}\label{Sobolev}
\|\partial_{s}v\|_{L^{p}} \leq C\|\partial^{\alpha}v\|_{L^{r}}^{\gamma}\| v\|_{L^{q}}^{1-\gamma},
\quad\text{where}\quad
\frac{1}{p} - \frac{s}{n} = \left(\frac{1}{r} - \frac{m}{n}\right) \gamma+ \frac{1}{q}(1-\gamma).
\end{align}

We note that the bilinear term 
% The following lemma, regarding the bilinear term 
\begin{align}
     B(u, v) := P_{\sigma}((u\cdot\nabla)v),\quad (u, v \in \mathcal{V}),
\end{align}
can be extended to a continuous map 
% (with a bit abuse of notation) 
$B : V \times V \to V'$,
\begin{align}
     \left<B(u, v), w\right> = \int_{\mathbb{T}^2}(u\cdot\nabla v)\cdot w\,dx, \quad (u, v, w\in \mathcal{V}).
\end{align}
In this work we will frequently use properties of the bilinear term which are contained for conciseness in the following lemma (see e.g., \cite{Constantin_Foias_1988,Temam_2001_Th_Num,Foias_Manley_Rosa_Temam_2001} for details).
\begin{lemma}
    \begin{subequations}
      \begin{align}
      \label{symm1}
      \ip{B(u,v)}{w}_{V'} &= -\ip{B(u,w)}{v}_{V'}, 
      \quad\text{ for all }\;u, v, w\in V,\\
      \label{symm2}
      \ip{B(u,v)}{v}_{V'} &= 0,
      \quad\text{ for all }\;u, v, w\in V.
    \end{align}
  \end{subequations}
  \begin{subequations}
    \begin{align}
      \label{B:424}
      |\ip{B(u,v)}{w}_{V'}|
      &\leq 
      C\normL{u}^{1/2} \normH{u}^{1/2} \normH{v} \normL{w}^{1/2} \normH{w}^{1/2},
%     \quad\text{ for all }\;\bu, \bv, \bw\in V,
      \\
      \label{B:442}
      |\ip{B(u,v)}{w}_{V'}|
      &\leq 
      C\normL{u}^{1/2} \normH{u}^{1/2} \normH{u}^{1/2} \normL{Av}^{1/2} \normL{w},
    %\quad\text{ for all }\;\bu\in V, \bv\in D(\bA), \bw\in H
      \\
      \label{B:inf22}
      |\ip{B(u,v)}{w}_{V'}|
      &\leq 
      C\normL{u}^{1/2} \normL{Au}^{1/2} \normH{v} \normL{w},
    \end{align}
for all $u$, $v$, $w$ in the largest spaces $H$, $V$, or $D(A)$, for which the right-hand sides of the inequalities are finite.
  \end{subequations}
  \begin{align}\label{enstrophy_miracle}
    \ip{B(w, w)}{Aw} = 0,\; \text{ for all } w\in D(A),
\end{align}
  and the Jacobi identity   
  \begin{align}\label{jacobi}
    \ip{B(u, w)}{Aw} + \ip{B(w, u)}{Aw} +\ip{B(w, w)}{Au} = 0.
  \end{align}
\end{lemma}

Next, we summarize a series of well-known results about \cref{NSE} in the next theorem and refer the readers to \cite{Constantin_Foias_1988, Temam_2001_Th_Num, Dascaliuc_Foias_Jolly_2005, Dascaliuc_Foias_Jolly_2007, Dascaliuc_Foias_Jolly_2008, Dascaliuc_Foias_Jolly_2010} for further details. 
We recall the dimensionless Grashof number $G$, defined as 
\begin{align} G := \frac{\|f\|_{L^2}}{\nu^2\lambda_1}, \end{align} 
where $\lambda_1$ is the first eigenvalue of the Stokes operator $A$. 
\begin{theorem}\label{thm1}
Let the initial data $u_0\in V$ and external forcing $f\in L^2_{\text{loc}}(0,T;H)$, then the two-dimensional Navier-Stokes system \cref{NSE} possesses a unique global solution \begin{align} 
% u\in L^{\infty}(0, T; V)\cap L^2(0, T; D(A)), 
u\in C(0, T; V)\cap L^2(0, T; D(A)), 
\end{align}
for any  positive $T < \infty$.

Moreover, for $f\in H$, then after certain time $t>t_0$, 
\begin{align}  \normL{u}\leq 2\nu^2 G^2 := M_0, \quad \normH{u}\leq 2\nu^2 \lambda_1 G^2 := M_1, \text{ and }\normL{\Delta u}\leq C\nu^2 \lambda_1^2 G^4 := M_2. \end{align}
\end{theorem}

\section{Main Results}

In this section we outline the main results of this work. 
We first show that the solutions, $v$, of \cref{NSE-v} are globally wellposed with regularity equivalent to that of the solution to the observed equation \cref{NSE} for initial data and external forcing that is sufficiently regular.
We next show exponential convergence in time in both $L^2$ and $H^1$ of the assimilated system to the observed system.
In our proof of exponential convergence in $L^2$ we consider only the case of Fourier modal projection as the spatial interpolant with a piece-wise constant time-extrapolation. We do this to fix ideas before considering the case of general interpolants in \Cref{thm-H1}.

We now proceed to proving the global well-posedness of \cref{NSE-v}. The proof this follows almost immediately from \Cref{thm1} for interpolants that satisfy \cref{interpolator_time}. 
As the extrapolation in time is explicit by assumption, at any time $t$  the feedback-control term can be considered an external force that does not depend on $v(t)$. We note that our convergence results are restricted to interpolants satisfying \cref{interpolator_time} however global well-posedness follows for any interpolant with an explicit time-extrapolation component so long as $I_{h,\kappa}(u-v)$ remains in $H$ for all time.

To simplify notation, in the below theorem and subsequent proof we use $I_{h,\kappa}^n:= I_{h,\kappa}\mid_{t_n \leq t < t_1}$, i.e. $I_{h,\kappa}^n(f)$ interpolates $f$ in space and extrapolates $f$ in time up to time $t_{n+1}$. We note that one can potentially utilize a $I_{h,\kappa}$ that changes depending on the value of $n$, for example one could opt to use all $n$ previous observations to create a $n-1$th order extrapolation. Such extrapolants pose no issue for global well-posedness so long as the resulting function remains in $L^2(t_n,t_{n+1};H)$. We note that for explicit interpolants $I_{h,\kappa}^n$, in particular those with a high order time-extrapolation component, can consider $I_{h,\kappa}^n(f)$ a function of $m$  previous observed values of $f$, i.e. $I_{h,\kappa}^n(f) := I_{h,\kappa}^n(f(t_0),f(t_1),...,f(t_m))$ for some $m \in \mathbb{N}$. To simplify notation we write $I_{h,\kappa}^n : X \to L^2(t_n,t_{n+1};H)$ with the understanding that $X$ is a Cartesian product of $V$ that potentially varies with $n$.

% \todo[inline]{Add paragraph discussing $n$ notation on the interpolant}

% We note that this assumption appears is essential to ensuring exponential convergence.
% This condition is required in \Cref{thm-H1} however we place additional bounds on $\kappa$, $\mu$, and $\mu\tau$ to obtain more explicit conditions for the parameters.
% In  \Cref{thm-L2} we opt to list this assumption explicitly

% \subsection{Well-posedness of Data Assimilation Algorithm}
% \todo[inline]{Note that wellposedness  should follow for arbitrary explicit interpolants, not just those satisfying \cref{interpolator_time}. I'm not sure how to phrase the spaces of the interpolant though.}

\begin{theorem}\label{thm1-AOT}
Let the initial data $v_0\in V$ and external forcing $f\in H$, then the 2D assimilated Navier-Stokes equations \cref{NSE-v} possess a unique global solution 
\begin{align} 
% v\in L^{\infty}(0, T; V)\cap L^2(0, T; D(A)), 
v\in C(0, T; V)\cap L^2(0, T; D(A)), 
\end{align} 
for any positive $T< \infty$ and any linear interpolant with explicit\footnote{Here, by explicit, we mean that $I_{h,\kappa}v$ can be computed for any given $v\in V$; i.e., with no dependence on future times $t>t_n$.} time-extrapolation $I_{h,\kappa}^n:X \to L^2(t_n,t_{n+1};H)$ with $t\in [t_{n},t_{n+1})$ for each $n \in \mathbb{N}\cup\set{0}$, where $X$ is a Cartesian product\footnote{representing the observations of previous times} of copies of $V$.
\end{theorem}
% \todo[inline]{Change $I_{h,\kappa}$ in this proof to $I_{h,\kappa}^n$.}

\begin{proof}
The proof of this follows from the global well-posedness of the 2D incompressible Navier-Stokes equations. 
In particular, by \Cref{thm1}, we know that for $u_0 \in V$ that $u \in L^\infty(0,T;V)\cap L^2(0,T;D(A))$.
Let us first consider a fixed time $t\in [t_0,t_0+\tau]$.
We know by assumption that $v_0 \in V$ and so $v_0 \in H$.
Now we set 
\begin{align}
    f_0 := f + \mu I_{h,\kappa}^0(u-v).
\end{align}
To apply \Cref{thm1} we require $f_0 \in L^2(t_0,t_0+\tau;H)$, which holds for $I_{h,\kappa}^0(u-v) \in L^2(t_0,t_0+\tau;H)$. 
Note here that for $I_{h,\kappa}^0$ an explicit time-extrapolant, $I_{h,\kappa}^0(u-v)$ depends only on previous states of $u-v$ that are known to be in $H$.
In particular we note that $I_{h,\kappa}^0(u-v)$ has no dependence on $v(t)$ or $v(t_1)$, for $t_0 < t \leq t_0 + \tau$.
Thus  $I_{h,\kappa}^0(u-v)$ is a function only of previous times.
Without loss of generality we will assume that $I_{h,\kappa}^0(u-v)$ depends only on $u_0$ and $v_0$ although in practice one could have several initial states of $u$ corresponding to times before $t_0$ with some chosen states to initialize $v$ with.
%the initial state, $u_0 - v_0 \in V$.
% This is true as by assumption we have that $I_{h,\kappa}^0$ is an explicit interpolant and therefore $f_i$ depends only on previous known states of $u-v$ from time before $t$. 
% As $I_{h,\tau}$ is an explicit interpolant $I_{h,\kappa}^0(u-v)$ is a function only of the previous known states of $u-v$ from the time before $t$. 
We know that $I_{h,\kappa}^0(u-v) \in L^2(t_0,t_0+\tau;H)$ by assumption, for regular data $u_0$ and $v_0$.
Thus we have $f_0 \in L^2(t_0,t_0+\tau;H)$.

Now, by \Cref{thm1}, $v\in C(t_0,t_0+\tau;V)\cap L^2(t_0,t_0+\tau;D(A))$.
% and thus we can now continue our solution $v$ to time $t_0+\tau$.
Moreover on the time interval $t\in [t_0+\tau, t_1]$ we note that the feedback-control term vanishes.
Therefore we can continue our solution from time $t_0+\tau$ up to time $t_1$ with $v \in C(t_0,t_1;V)\cap L^2(t_0,t_1;D(A))$.

Iterating this procedure with $f_i = f + I_{h,\kappa}^i(u-v)$  we obtain $v\in C(t_0,T;V)\cap L^2(0,T;D(A))$.
Additionally we note that also by \cref{thm1} that all of the extensions of $v$ to time $t_i$ produced by this procedure are unique.
\end{proof}

To prove exponential convergence of $v$ to $u$ in $L^2$, we first consider the case when $I_{h,\kappa}(f(t)) :=  P_m(f({t_n}))$, for $t \in [t_{n},t_{n+1})$, i.e., the projection onto the first $m$ Fourier modes with piecewise constant extrapolation in time. 
This particular case was studied by itself in \cite{Foias_Mondaini_Titi_2016} for $\tau = \kappa$, so we similarly look at this simplified choice of interpolants to fix the ideas before tackling the general case.
In the general setting we extend our results for general interpolants that satisfy \cref{interpolator_time}. It is worth noting that the conditions on our interpolators restrict us to zero-th order time-extrapolants, which is piece-wise constant extrapolation. Our results can be extended to interpolants with a higher-order time extrapolation, including piece-wise linear extraplation, which require estimates on norms of higher order time derivatives of $u$ and $v$. We note that such a generalization is unwarranted as, in computations we can simply decrease the value of $\tau$ to be small enough such that the extrapolation will be approximately constant while increasing $\mu$ to compensate.

We note that the requirements of \Cref{thm-L2} are similar to those in \cite{Foias_Mondaini_Titi_2016} for observational data that is noise-free, however we require the extra assumption \cref{decay assumption}. 
This estimate seems necessary in order to obtain exponential decay in $L^2$. 
We note that in the case of $\tau = \kappa$, \cref{decay assumption} is automatically satisfied for a sufficient constant $C>0$.
%%% Note: Foias Mondaini Titi state this in their paper, but do not specify $C$.
It is worth mentioning that one could relax the conditions on $\kappa$ in \cref{kappa-bound} and instead require such a bound only on $\tau$.
We place this restriction on $\kappa$ instead as we want such a condition to hold for arbitrary choice of $\tau \leq \kappa$.
%%% Note: putting the restriction of $\tau$ instead should work, but then we need some bound on $\kappa - \tau$. The value of $\kappa$ shouldn't be restricted by $\tau$ if we can avoid it.

% We also note that one can obtain results similar to \Cref{thm-H1,thm-L2} by specifying $ $
% , as in \cite{Foias_Mondaini_Titi_2016}.

% with a similar inequality \cref{theta-H1} appearing in the proof of $H^1$ convergence. 
% The requirements of \Cref{thm-H1} are markedly different from the corresponding theorem in \cite{Foias_Mondaini_Titi_2016}, which follows from the fact that we did not use the Br\'ezis-Gallouet inequality in the estimate of the nonlinear term. 

% \subsection{Convergence to True Solution}
% In this section we outline our theorems regarding the asymptotic behavior of \cref{NSE-v}. 
% By explicit, we mean that $I_{h,\kappa}(f(t))$ is given as a function of $t$ and $f(t_i)$ for $t_i \leq n$ with $t\in [t_n, t_{n+1})$, such as Adams-Bashforth and explicit Euler schemes. 
% Implicit time-extrapolants, where $I_{h,\kappa}(f(t))$ may depend on the future value of $f$, e.g. $I_{h,\kappa}(f(t)) = f(t_{n+1})$, present a difficulty particularly for the well-posedness of such schemes using the AOT algorithm due to the requirement of future knowledge of $v$.

\begin{theorem}\label{thm-L2}
Let $u$ and $v$ be solutions to \cref{NSE} and \cref{NSE-v}, respectively, with initial data $u_0, v_0 \in V$ and $v$ defined on $[t_0,\infty)$.
Suppose $I_{h,\kappa}(f(t)) :=  P_m(f({t_n}))$, for $t \in [t_{n},t_{n+1})$, and the following conditions are satisfied with $0<\tau \leq \kappa$,
\begin{align} \lambda_{m+1} \geq 6\frac{\mu}{\nu}, \end{align}

\begin{align} \mu  \geq C \frac{M_1^2}{\nu}, \end{align}
% \begin{align} \mu \tau < \frac{1}{K}\end{align} %Redundant with next line
\begin{align}\left( e^{-\frac{\mu}{2}\tau} + K\mu\tau (1- e^{-\frac{\mu}{2}\tau})\right)e^{\frac{cM_1^2}{\nu}\kappa} < 1, \label{decay assumption} \end{align}

 \begin{align} \label{kappa-bound}
 \kappa  \leq \frac{C}{\mu}\min{\Bigg \{} & 1, \frac{\nu}{R}, \frac{\nu^2}{R^2}, \frac{\nu^{3/2}\mu^{1/2}}{M_0M_1},
 \frac{\nu^2\lambda_1^{1/2}}{M_0M_1}, \frac{\nu^{1/2}\lambda_1^{1/2}}{\mu^{1/2}}, \frac{\nu^2 \lambda_1^2}{\mu^2}, \frac{\nu \mu}{M_1^2}\Bigg\},
 \end{align}
 then $\normL{u-v}$ decays to zero exponentially fast in time.
 % for any choice of $0< \tau \leq \kappa$,

 Here $C>0$ is an absolute constant and $K>0 $, depending only on $\nu, \lambda_1, M_0, M_1$ and  absolute constants, is specified in \cref{constant:K}. 
\end{theorem}

\begin{remark}
We note that the conditions above are not vacuous, although the requirements on $\mu$, $\tau$, and $\kappa$ depend on one another. Condition \cref{decay assumption} yields explicit bounds on $\kappa$ by fixing a value of $\mu\tau$.  For instance, setting $\mu\tau = \frac{1}{2K}$ yields
\begin{align}
e^{-\frac{\mu}{2}\tau} + K\mu\tau (1- e^{-\frac{\mu}{2}\tau}) = \frac{1}{2}\left( 1 + e^{-\frac{\mu}{2}\tau} \right). \label{not-theta}
\end{align} 
% Choosing $\mu\tau = \frac{1}{2K}$ yields a value close to the minimum value of this expression for large values of $K$.
% The particular which is close to the minimum value of this expression.
% We note that the true minimum can be extracted in terms of the Lambert W function, but doing so is unnecessary as we require only to bound \cref{not-theta} away from 1 in order for \cref{decay assumption} to hold.
which in turn yields the following smallness condition on $\kappa$:
\begin{align}
    \kappa < -\frac{\nu}{cM_1^2}\ln{\left(\frac{1}{2}\Big(1+e^{-\frac{1}{4K}}\Big)\right)}.
\end{align}

We note that the particular value of \cref{not-theta} is unimportant so long as it is bounded away from $1$.
It is worth noting that the minimum value of of \cref{not-theta} can be extracted in terms of the Lambert-W function and choosing $\mu\tau$ to attain this minimum will provide the steepest decay.
However this also will require the most restrictive smallness condition for $\kappa$.

\end{remark}

\begin{proof}
The proof this follows similarly from that of the case $\tau = \kappa$ as found in \cite{Foias_Mondaini_Titi_2016} in the absence of stochastic noise, though one has to work differently on the intervals $[t_n, t_n+\tau)$ and $[t_n+\tau, t_{n+1})$.
Denoting by $w = v-u$, subtracting \cref{NSE} from \cref{NSE-v}, we obtain the following equation for $w$:
\begin{align}
    \label{NSE-w}
    \frac{dw}{dt} + \nu Aw + B(w,u) + B(u,w) + B(w,w) = -\mu \sum_{n=0}^\infty P_m w(t_n)\chi_{n,\tau}.
\end{align}
Multiplying \cref{NSE-w} by $w$ and integrating over $\Omega$, yields
\begin{align}\label{L2-estimate}
    \frac{1}{2}\frac{d}{dt}\normL{w}^2 
    + 
    \nu \normH{w}^2 
    &= 
    -
    \braket{B(w,u),w}_{V',V} 
    - 
    \mu \sum_{n=0}^\infty \normL{P_mw}^2\chi_{n,\tau}
    \\&\quad\notag
    - 
   \mu \sum_{n=0}^\infty (P_m(w(t_n)-w),w)_{L^2}\chi_{n,\tau}.
\end{align}

We now move to estimate the individual terms on the right side of the above equation. 
To bound the first term we use Sobolev inequalities, Young's inequality, as well as \Cref{thm1}
\begin{align}
    \abs{\braket{B(w,u),w}_{V',V}} &\leq c\normH{u}\normL{w}\normH{w}\\
    &\leq \tfrac{\nu}{6}\normH{w}^2 + \tfrac{c}{\nu}\normH{u}^2\normL{w}^2\\
    &\leq \tfrac{\nu}{6}\normH{w}^2 + c\tfrac{M_1^2}{\nu}\normL{w}^2.
\end{align}
Here and henceforth, $c$ denotes an absolute constant that may change from line to line depending only on the dimension and the domain of the problem.

Regarding the second term we find
% , with $Q_m$ denoting the Fourier projection onto modes larger than $m$, i.e. $Q_m u = u - P_m u$, we find 
\begin{align}
    -\mu \normL{P_mw}^2 
    &= 
    -
    \mu \Big(\normL{w}^2 
    - 
    \normL{Q_m w}^2\Big)
    \\&= 
    -
    \mu \normL{w}^2 
    + 
    \mu \normL{Q_m w}^2
    \\&
    \leq 
    -
    \mu \normL{w}^2 
    + 
    \tfrac{\mu}{\lambda_{m+1}}\normH{Q_m w}^2
    \\&
    \leq
    -
    \mu \normL{w}^2 
    + 
    \tfrac{\nu}{6}\normH{w}^2.
    %%%%%%%%%%%%%%%% HERE is where \lambda_{m+1} assumption is used.
\end{align}
Here $Q_m u = u - P_m u$ is the projection onto all of the Fourier modes larger than $m$.
This in turn implies 
\begin{align}  -\mu \sum_{n=0}^\infty \normL{P_mw}^2\chi_{n,\tau} \leq \sum_{n=0}^\infty \left( -\mu \normL{w}^2 + \frac{\nu}{6}\normH{w}^2\right)\chi_{n,\tau}.  \end{align}

As for the last term, we estimate
\begin{align}
    \mu \abs{(P_m(w(t_n)-w(t),w(t))_{L^2}} &= \mu \abs{ (w(t_n)-w(t),P_m w(t))_{L^2} }\notag\\
    &= \mu \abs{ \left(\int_{t_n}^t \pd{w}{s} ds, P_mw(t) \right)_{L^2}  }\notag\\
    &\leq \mu \left( \int_{t_n}^t \norm{\pd{w}{s}(s)}_{V'} ds \right)\normH{w(t)}\notag\\
    &\leq \frac{\nu}{6}\normH{w(t)}^2 + \frac{c\mu^2}{\nu}\left( \int_{t_n}^t \norm{\pd{w}{s}(s)}_{V'} ds \right)^2.\label{3rd-term-time}
\end{align}

In order to bound this term we now require estimates on the time-derivative $\norm{\pd{w}{s}}_{V'}$, for which we now work directly with \cref{NSE-w} to obtain. For fixed $t\in [t_n,t_n+\tau)$, it follows that
\begin{align}%\small
    \norm{\pd{w}{s}(t)}_{V'} \label{V'Bound} 
    &
    \leq 
    \nu \normH{w} 
    + 
    \norm{B(w,u)}_{V'} 
    + 
    \norm{B(u,w)}_{V'} 
    \\&\quad\notag
    + 
    \norm{B(w,w)}_{V'} 
    + 
    \mu \sum_{n=0}^\infty\norm{P_m w(t_n)}_{V'}\chi_{n,\tau}
    \\&
    \leq \notag
    \nu \normH{w} 
    + 
    cM_0^{1/2}M_1^{1/2}\normL{w}^{1/2}\normH{w}^{1/2} 
    + 
    c\normL{w}\normH{w} 
    \\&\quad \notag
    + 
    \mu \int_{t_n}^t \norm{\pd{w}{s}(s)}_{V'}ds 
    + 
    \frac{\mu}{\lambda_1^{1/2}}\normL{w}.
\end{align}

% Note that if we instead fix $t\in [t_n, t_{n}+\tau]$ we still maintain this upper bound:
% Thus, we obtain the following estimates:
% \begin{align}\label{V'Bound} 
%   \norm{\pd{w}{s}(t)}_{V'} 
%   &\leq 
%   \nu \normH{w} 
%   + 
%   cM_0^{1/2}M_1^{1/2}\normL{w}^{1/2}\normH{w}^{1/2} 
%   + 
%   c\normL{w}\normH{w} 
%   \\&\quad
%   + 
%   \mu \int_{t_n}^t \norm{\pd{w}{s}(s)}_{V'}ds 
%   + 
%   \frac{\mu}{\lambda_1^{1/2}}\normL{w} \nonumber.
% \end{align}

Integrating from $t_n$ to $t\in [t_n,t_{n}+\tau]$ yields
\begin{align}
    \int_{t_n}^{t} \norm{\pd{w}{s}(t)}_{V'} ds 
    &
    \leq 
    \int_{t_n}^t 
    \Big(\nu \normH{w} 
    + 
    cM_0^{1/2}M_1^{1/2}\normL{w}^{1/2}\normH{w}^{1/2} \Big.
    + 
    \\&\quad\notag
    \Big.
    c\normL{w}\normH{w}  + \frac{\mu}{\lambda_1^{1/2}}\normL{w}\Big) ds
    + 
    \mu \tau \int_{t_n}^t \norm{\pd{w}{s}(s)}_{V'}ds, 
\end{align}
where we used the following fact
\begin{equation}
    \int_{t_n}^t\int_{t_n}^s \norm{\frac{dw}{dz}(z)}dzds 
    \leq 
    \int_{t_n}^t \int_{t_n}^t \norm{\frac{dw}{dz}(z)}_{V'}dzds 
    \leq 
    \tau \int_{t_n}^t \norm{\pd{w}{s}(s)}_{V'}ds.
\end{equation}

Observing our assumptions $\tau \leq \kappa \leq \frac{c}{\mu}$ we obtain the following estimate:
\begin{align}
    \int_{t_n}^t \norm{\pd{w}{s}}_{V'} ds 
    &\leq 
    c \int_{t_n}^t \nu \normH{w} 
    + 
    cM_0^{1/2}M_1^{1/2}\normL{w}^{1/2}\normH{w}^{1/2} 
    \\&\quad\notag
    + 
    \normL{w}\normH{w} 
    + 
    \frac{\mu}{\lambda_1^{1/2}}\normL{w} ds.
\end{align}

Squaring both side and applying H\"older's inequality now yields the estimate we require for \cref{3rd-term-time}:
\begin{equation}
    \left(\int_{t_n}^t \norm{\pd{w}{s}}_{V'} ds\right)^2 
    \leq 
    c\tau^2 \int_{t_n}^t \phi(s)ds,
\end{equation}
where 
\begin{equation}
    \phi(s) 
    := 
    \nu^2 \normH{w}^2 
    + 
    M_0M_1\normL{w}\normH{w} 
    + 
    \normL{w}^2\normH{w}^2 
    + 
    \frac{\mu^2}{\lambda_1}\normL{w}^2.
\end{equation}
%Note that the \kappa comes from applying Holder's inequality (\int_{t_n}^t 1 ds \leq \kappa)
Combining all the above estimates in \cref{L2-estimate}, we obtain 
\begin{equation}\label{EQ}
    \frac{d}{dt}\normL{w}^2 
    + 
    \nu \normH{w}^2 
    \leq 
    -
    \mu\sum_{n=0}^\infty \normL{w}^2\chi_{n,\tau} 
    + 
    \frac{cM_1^2}{\nu}\normL{w}^2 
    + 
    \frac{c\mu^2\tau^2}{\nu}\sum_{n=0}^\infty\chi_{n,\tau}(t)\int_{t_n}^t\phi(s)ds.
\end{equation}

Now we note that by \Cref{thm1} the following must be true:
\begin{align}
     \normL{w(t_0)} \leq \normL{u(t_0)} + \normL{v(t_0)} \leq 2M_0:= R. 
\end{align}
Here we are assuming without loss of generality that $\normL{v(t_0)}\leq M_0$.
% We note without loss of generality we can assume that $\normL{v(t_0)}\leq M_0$, as otherwise 
As $w\in C([t_0,\infty); H)$, there must exist an $s \in [t_0,\infty)$ such that $\normL{w(t)} \leq 2R,\,\,\text{ for all } t\in [t_0,s]$.

Denote 
\begin{align} \widetilde{t} := \sup\{s \in [t_0, \infty): \sup_{t\in[t_0,s]}\normL{w(t)}< 2R\}. \end{align} 

We now show that  $\widetilde{t} > t_0 + \tau $. Supposing that $\widetilde{t}\leq t_0+\tau$
we integrate \cref{EQ} in time from $t_0$ to $t \leq \widetilde{t}$:
\begin{align}
    &
    \normL{w(t)}^2 
    - 
    \normL{w(t_0)}^2 
    + 
    \nu \int_{t_0}^t \normH{w(s)}^2 ds 
    \\&\quad
    \leq 
    -
    \left(\mu - \frac{cM_1^2}{\nu}\right)\int_{t_0}^t\normL{w(s)}^2ds 
    + 
    \frac{c\mu^2\tau^2}{\nu}\int_{t_0}^t \phi(s)ds.\nonumber 
\end{align}

Then, by Young's inequality and the assumption $\normL{w(s)}\leq 2R$ for all $s\in [t_0,t]$, we obtain
\begin{align}
    \frac{c\mu^2\tau^2}{\nu} \int_{t_0}^t \phi(s)ds 
    &
    = 
    \frac{c\mu^2\tau^2}{\nu}\int_{t_0}^t \nu^2\normH{w}^2 
    + 
    M_0M_1 \normL{w}\normH{w} 
    \\&\quad\notag
    + 
    \normL{w}^2\normH{w}^2 
    + 
    \frac{\mu^2}{\lambda_1}\normL{w}^2 ds
    \\&
    \leq \notag
    \frac{c\mu^2\tau^2}{\nu}\int_{t_0}^t \nu^2\normH{w}^2 
    + 
    \frac{M_0M_1}{\nu^2}\normL{w} 
    \\&\quad\notag
    + 
    R^2\normH{w}^2 
    + 
    \frac{\mu^2}{\lambda_1}\normL{w}^2 ds. 
\end{align}
After rearranging, we obtain
\begin{align}
    &
    \normL{w(t)}^2 
    - 
    \normL{w(t_0)}^2 
    + 
    \nu \Big(1-c\mu^2\tau^2(1+\frac{R^2}{\nu^2})\Big)\int_{t_0}^t \normH{w(s)}^2ds
    \\&\quad \notag
    \leq 
    -
    \left(\mu - \frac{cM_1^2}{\nu} - \frac{c\mu^2\tau^2M_0^2M_1^2}{\nu^3} - \frac{c\mu^4\tau^2}{\lambda_1\nu}\right)\int_{t_0}^t \normL{w}^2 ds.
\end{align}

Next, for sufficiently small $\tau$ and large enough $\mu$ 
\begin{equation}
    1-c\mu^2\tau^2(1+\frac{R^2}{\nu^2}) 
    \geq 
    \frac{1}{2}
\end{equation}
and
\begin{equation}
    -\left(\mu - \frac{cM_1^2}{\nu} - \frac{c \mu^2 \tau^2 M_0^2M_1^2}{\nu^3} - \frac{c\mu^4\tau^2}{\lambda_1\nu}\right)
    \leq 
    -\frac{\mu}{2},
\end{equation}
which further implies
\begin{equation}
    \normL{w(t)}^2 
    - 
    \normL{w(t_0)}^2 
    + 
    \frac{\nu}{2} \int_{t_0}^t \normH{w(s)}ds
    \leq
    -
    \frac{\mu}{2}\int_{t_0}^t \normL{w}^2 ds.
\end{equation}
Note that from the above inequality, 
\begin{equation}\label{H1_integral_bound}
    \int_{t_0}^t \normH{w(s)}^2ds 
    \leq 
    \frac{2}{\nu} \normL{w(t_0)}^2.
\end{equation}
We then return to \cref{EQ} and apply Poincar\'e's inequality in order to bound all the terms in the definition of $\phi$ in terms of $\normH{w}^2$. Namely, we obtain
\begin{equation}
    \dv{}{t}\normL{w}^2 
    \leq 
    -
    \Big(\mu - \frac{cM_1^2}{\nu}\Big)\normL{w}^2 
    + 
    \frac{c\mu^2\tau}{\nu}\left( \nu^2 + \frac{M_0M_1}{\lambda_1^{1/2}} 
    + 
    R^2 \frac{\mu^2}{\lambda_1^2} \right)\int_{t_0}^t\normH{w}^2. \\
\end{equation}

Again, utilizing our smallness condition on $\mu$ we obtain
\begin{equation}
    \dv{}{t}\normL{w}^2 
    \leq 
    -
    \frac{\mu}{2}\normL{w}^2 
    + 
    \frac{c\mu^2\tau}{\nu}\left( \nu^2 + \frac{M_0M_1}{\lambda_1^{1/2}} 
    + 
    R^2 \frac{\mu^2}{\lambda_1^2} \right)\int_{t_0}^t\normH{w}^2. \\
\end{equation}
Note that the above equation is only valid for $t\in [t_0,\widetilde{t}\,]$.
Now, using \cref{H1_integral_bound} we obtain the following inequality
\begin{equation}
    \dv{}{t}\normL{w}^2 
    \leq
    -
    \frac{\mu}{2}\normL{w}^2 
    + 
    \frac{c\mu^2\tau}{\nu}\left( \nu^2 + \frac{M_0M_1}{\lambda_1^{1/2}} + R^2 \frac{\mu^2}{\lambda_1^2} \right)\int_{t_0}^t\frac{2}{\nu}\normL{w(t_0)}^2.
\end{equation}

Again we integrate this equation from time $t_0$ to $t$, which yields
\begin{align}
    \normL{w(t)}^2 
    &
    \leq 
    \normL{w(t_0)}^2\left( e^{-\frac{\mu}{2}(t-t_0)} \right.
    \\&\quad \notag
    \left.
    + c\mu \tau\Big(1+ \frac{M_0M_1}{\nu^2\lambda_1^{1/2}} + \frac{R^2}{\nu^2} + \frac{\mu^2}{\nu^2\lambda_1^2}\Big)\left(1- e^{-\frac{\mu}{2}(t-t_0)} \right) \right)
    \\&
    \leq \notag
    R^2 \left( e^{-\frac{\mu}{2}(t-t_0)} + K\mu\tau (1- e^{-\frac{\mu}{2}(t-t_0)})\right), 
\end{align}
where \begin{align} K := c\Big(1+ \frac{M_0M_1}{\nu^2\lambda_1^{1/2}} + \frac{R^2}{\nu^2} + \frac{\mu^2}{\nu^2\lambda_1^2}\Big).\label{constant:K} \end{align}
We also note 
\begin{equation}
    \left( e^{-\frac{\mu}{2}(t-t_0)} + K\mu\tau(1- e^{-(\frac{\mu}{2})(t-t_0)})\right) < 1,
\end{equation}
holds for any choice of $t\leq \tau$ so long as 
\begin{align}
    \mu\tau < \frac{1}{K},
\end{align}
holds. This follows from \cref{decay assumption} and  thus, we proved 
\begin{align} \normL{w(t)} \leq R\text{ for all } t\in [t_0,\widetilde{t}\,], \end{align} 
implying $\widetilde{t} > t_0+\tau$. 
In particular, we obtain 
\begin{equation}
    \normL{w(t_0+\tau)}^2 \leq \normL{w(t_0)}^2\theta,
\end{equation}
where
\begin{align}
   \theta := \left( e^{-\frac{\mu}{2}\tau} + K\mu\tau (1- e^{-\frac{\mu}{2}\tau})\right) < 1.
\end{align}

Next, we follow a similar argument for $\widetilde{t}$ in the time interval $[t_0+\tau, t_1)$. We start by considering \cref{EQ} on the time interval $[t_0+\tau, t]$, where $t\in[t_0+\tau, \widetilde{t}\,]$, i.e., 
\begin{align}
    \dv{}{t}\normL{w}^2 
    &
    \leq 
    \frac{cM_1^2}{\nu}\normL{w}^2 - \nu \normH{w}^2
    \leq 
    \frac{cM_1^2}{\nu}\normL{w}^2.
\end{align}
Integrating the above inequality in time from $t_0+\tau$ to $t$ yields:
\begin{align}
    \normL{w(t)}^2 
    &
    \leq 
    \normL{w(t_0+\tau)}e^{\frac{cM_1^2}{\nu}(t-(t_0+\tau))}
    \\&
    \leq \notag
    \normL{w(t_0)}^2\theta e^{\frac{cM_1^2}{\nu}(t-(t_0+\tau))}
    \\&\notag
    \leq 
    \normL{w(t_0)}^2\theta e^{\frac{cM_1^2}{\nu}(\kappa - \tau)}
    \\&\notag
    \leq 
    R^2\theta e^{\frac{cM_1^2}{\nu}\kappa}
\end{align}

Now we note that in order to obtain exponential decay we require that 
\begin{align}
    \theta e^{\frac{cM_1^2}{\nu}\kappa} < 1,
\end{align}
which holds by \cref{decay assumption}.

Thus $\norm{w(t)}_{L^2}^2 \leq R^2$ for any time $t\leq t_0 + \kappa$, and therefore $\widetilde{t}> t_0 +\kappa$. We can iterate our argument to obtain that 

\begin{align}
\norm{w(t_n)}_{L^2}^2 \leq \norm{w(t_0)}_{L^2}^2 \sigma^n    
\end{align}
where 
\begin{align}\sigma := \theta e^{\frac{cM_1^2}{\nu}\kappa} < 1. \end{align}
Thus we obtain exponential convergence of $v$ to $u$ in $L^2$.

\end{proof}

We now consider the case of general interpolants $I_{h,\kappa}$ with a first-order explicit time-extrapolation satisfying \cref{interpolator_time}. 
% It is worth noting that the condition of \Cref{thm-H1} differ significantly from the equivalent statement in \cite{Foias_Mondaini_Titi_2016}. We note that this is mostly due to the
% Readers may note that \cref{decay assumption}, which we stated was necessary, appears to be absent from the requirements of \Cref{thm-H1}.
We note that while the inequality \cref{decay assumption} appears to be missing from the theorem statement, a similar statement is necessary (see \cref{theta-H1}). For the case of general interpolants we opted to satisfy \cref{theta-H1} by incorporating explicit bounds on $\mu$, $\kappa$, and $\mu\tau$.

%Here mention how 

\begin{theorem}\label{thm-H1}
Let $u$ and $v$ be solutions to \cref{NSE} and \cref{NSE-v}, respectively, with initial data $u_0, v_0 \in V$ and $v$ defined on $[t_0,\infty)$.
Suppose $I_{h,\kappa}$ satisfies \cref{interpolator_time}, and the following conditions are satisfied with $0<\tau \leq \kappa$, 
% $\normH{u-v}$ approaches zero exponentially fast in time for any choice of $0< \tau \leq \kappa$, and an absolute constant $C$,
\begin{align}
    \mu 
    &
    > 
    \max\left\{ \frac{CM_1M_2}{\nu\lambda_1^{1/2}},\,\, \frac{CM_0^2M_1^2}{\nu^3},\,\, \frac{CM_1M_2}{\nu\lambda_1^{1/2}} + \frac{CM_1^4}{\nu^3\lambda_1}, 4K_2 \right\},\\
    h 
    &
    < 
    \frac{1}{2}\sqrt{\frac{\nu}{c_0\mu}},\\
    \kappa \label{kappa-bound-H1}
    &
    < 
    \min\left\{ \frac{e^{-K_2}}{4C\mu^2 K_1}, \frac{C\lambda_1^{1/2}\nu^{1/2}}{\mu^{3/2}}, \frac{\frac{3K_2^2}{2} + \mu^2  K_1 - \sqrt{D}}{K_2^3}\right\},\\&\text{where}\,\, D=\Big(\frac{3K_2^2}{2} + \mu^2  K_1\Big)^2 - 2K_2^4>0,\notag\\
    \text{and}\quad\mu\tau
    &
    \geq 
    4K_2\kappa \label{mu-tau-log bound} ,
\end{align}
 then $\normH{u-v}$ decays to zero exponentially fast in time.

Here $C$ denotes an absolute constant, and $K_1,  K_2>0$, depending only on $\nu, \lambda_1, M_0, M_1, M_2$ and some absolute constants, are specified in \cref{K-1} and \cref{K-2}. 

\end{theorem}
\begin{proof}
Subtracting \cref{NSE} from \cref{NSE-v}, and denoting $w:=v-u$, we obtain 
  \begin{align}\label{NSE-diff}
    \frac{dw}{dt} + \nu Aw + B(w,u) + B(u,w) + B(w,w) = -\mu \sum_{n=0}^\infty P_{\sigma}I_{h,\kappa}(w(t_n))\chi_{n,\tau}.
  \end{align}

We multiply \cref{NSE-diff} by $Aw = -\Delta w$, and integrate by parts on $\Omega$, to obtain 
    \begin{align}\label{H1-estimate}
      & \quad
      \frac{1}{2}\frac{d}{dt}\normH{w}^2 
      + 
      \nu\normL{\Delta w}^2
      +
            \mu \sum_{n = 0}^\infty\norm{w}^2_{H^1}\chi_{n,\tau}
      \\&\notag
      = 
      (B(w,u),\lap w)_{L^2} 
      + 
      (B(u,w),\lap w)_{L^2} 
      - 
      \mu \sum_{n = 0}^\infty(w - I_{h,\kappa}(w), \Delta w)_{L^2}\chi_{S_{\tau}} 
    \end{align}
  where we used \cref{B:inf22}. 
We note that the summation of feedback-control terms will have at most one non-zero term for any given value of $t$, thus we omit the summation of these terms to simplify our notation moving forwards.
  
Next, we estimate the two nonlinear terms given on the right side of \cref{H1-estimate}.
For the first term, using H\"older's and Poincar\'e inequalities, as well as \Cref{thm1}, we estimate
\begin{align}
  |(B(w,u), \Delta w)_{L^2}|
  &\leq
  \fournorm{w}\fournorm{\nabla u}\normL{\Delta w}
  \leq
  \normL{w}^{1/2}\normH{w}^{1/2}\normH{u}^{1/2}\normL{\Delta u}^{1/2}\normL{\Delta w}
  \\&\notag
  \leq
  \sqrt{M_1 M_2}\frac{\normH{w}}{\lambda_1^{1/4}}\normL{\Delta w}
  \leq
  \frac{c M_1 M_2}{\nu\lambda_1^{1/2}}\normH{w}^2 + \frac{\nu}{16}\normL{\Delta w}^2.
\end{align}

Performing a similar procedure on the second term we obtain the following estimate:
    \begin{align}
      |(B(u,w), \Delta w)_{L^2}|
      &\leq
      \fournorm{u}\fournorm{\nabla w}\normL{\Delta w}
      \leq
      \normL{u}^{1/2}\normH{u}^{1/2}\normH{w}^{1/2}\normL{\Delta w}^{3/2}
      \\&\notag
      \leq
      \sqrt{M_0 M_1}\normH{w}^{1/2}\normL{\Delta w}^{3/2}
      \leq
      \frac{c M_0^2 M_1^2}{\nu^3}\normH{w}^2 + \frac{\nu}{16}\normL{\Delta w}^2.
    \end{align}

In order to estimate the remaining terms, we apply H\"older's inequality, Young's inequality, as well as \cref{interpolator_time}. This yields the following
\begin{align}
    -\mu (w - I_{h,\kappa} w , \Delta w) &\leq \frac{\mu^2 c_0 h^2}{\nu} \norm{w}_{H^1}^2 + \frac{c\mu^2 \tau }{\nu}\int_{t_n}^t\norm{\frac{dw(s)}{ds}}^2_{L^2} ds + \frac{\nu}{4}\norm{\Delta w}_{L^2}^2\\
    &\leq\notag \frac{\mu}{4} \norm{w}_{H^1}^2 + \frac{c\mu^2\tau }{\nu}\int_{t_n}^t\norm{\frac{dw(s)}{ds}}^2_{L^2} ds + \frac{\nu}{4}\norm{\Delta w}_{L^2}^2.
\end{align}

We thus need to estimate the above time-derivative. For fixed $t\in[t_n, t_n+\tau)$, using \cref{NSE-diff}, we estimate 
  \begin{align}
  \allowdisplaybreaks\label{time_derivative}
    &\quad
    \left\Vert\frac{dw}{ds}\right\Vert_{L^2}
    \\&\notag
    \leq
    \nu\normL{\Delta w} + \normL{B(w, u)} + \normL{B(u, w)} + \normL{B(w, w)}
    \\&\quad\notag
    +\mu\normL{I_{h,\kappa}(w)- w} + \mu\normL{w} 
    \\&\notag
    \leq
    \nu\normL{\Delta w(s)} + \normL{B(w(s), u(s))} + \normL{B(u(s), w(s))} + \normL{B(w(s), w(s))}
    \\&\quad\notag
    +\mu\normL{I_{h,\kappa}(w)- w} + \mu\normL{w} 
    \\&\notag
    \leq
    \nu\normL{\Delta w(s)} 
    + 
    c\normL{w(s)}^{1/2}\normH{w(s)}^{1/2}\normH{u(s)}^{1/2}\normL{\Delta u(s)}^{1/2}
    \\&\quad\notag
    +
    c\normL{u(s)}^{1/2}\normH{u(s)}^{1/2}\normH{w(s)}^{1/2}\normL{\Delta w(s)}^{1/2}
    +
    c\normL{w(s)}^{1/2}\normH{w(s)}\normL{\Delta w(s)}^{1/2}
    \\&\quad    \notag
    +\mu\normL{I_{h,\kappa}(w)- w} + \mu\normL{w} 
    \\&\notag
    \leq
    \nu\normL{\Delta w(s)} 
    +
    \frac{c\sqrt{M_1M_2}}{\lambda_1^{1/4}}\normH{w(s)}
    +
    c\sqrt{M_0M_1}\normH{w(s)}^{1/2}\normL{\Delta w(s)}^{1/2}
    \\&\quad\notag
    +
    \frac{c}{\lambda_1^{1/4}}\normH{w(s)}^{3/2}\normL{\Delta w(s)}^{1/2}
    +\mu\normL{I_{h,\kappa}(w)- w} + \mu\normL{w} 
    \\&\notag
    \leq
    \nu\normL{\Delta w(s)} 
    +
    \frac{c\sqrt{M_1M_2}}{\lambda_1^{1/4}}\normH{w(s)}
    +
    c\sqrt{M_0M_1}\normH{w(s)}^{1/2}\normL{\Delta w(s)}^{1/2}
    \\&\quad\notag
    +
    \frac{c}{\lambda_1^{1/2}\nu}\normH{w}^3 + \nu\normL{\Delta w}
    +\mu\normL{I_{h,\kappa}(w)- w} + \mu\normL{w}.
  \end{align}

Now, squaring both sides of the inequality above, we obtain:
\begin{align}
    \left\Vert\frac{dw}{ds}\right\Vert_{L^2}^2
    &\leq %\notag
    4\nu^2\normL{\Delta w(s)}^2 
    +
    \frac{cM_1M_2}{\lambda_1^{1/2}}\normH{w(s)}^2
    +
    cM_0M_1\normH{w(s)}\normL{\Delta w(s)}
    \\&\quad\notag
    +
    \frac{c}{\lambda_1\nu^2}\normH{w}^6
    +\mu^2\normL{I_{h,\kappa}(w)- w}^2 + \mu^2\normL{w}^2
    \\
    &\leq \notag
    5\nu^2\normL{\Delta w(s)}^2 
    +
    \left(\frac{cM_1M_2}{\lambda_1^{1/2}} +\frac{cM_0^{2}M_1^{2}}{\nu^2}  +\frac{\mu \nu}{4} + \frac{\mu^2}{\lambda_1} + \frac{c}{\lambda_1\nu^2} \normH{w}^4\right) \normH{w(s)}^2
    \\&\quad\notag
    +
    c\mu^2\tau \int_{t_n}^t\norm{\frac{dw(s)}{ds}}_{L^2}^2 ds, 
\end{align}
where in the second inequality above, we used the interpolation inequality
 \ref{interpolator_time} as well as both Young's and Poincar\'e's inequalities.

Integrating the above equation in time with respect to $s$ from $t_n$ to $t \in [t_n, t_n+\tau)$ yields
\begin{align}
&\quad
    \int_{t_n}^t \left\Vert\frac{dw}{ds}\right\Vert_{L^2}^2 ds
    \\&\leq  \notag
    \int_{t_n}^t \Big(  5\nu^2\normL{\Delta w(s)}^2 
    +
    \left(\frac{cM_1M_2}{\lambda_1^{1/2}} +\frac{cM_0^{2}M_1^{2}}{\nu^2}  +\frac{\mu \nu}{4} + \frac{\mu^2}{\lambda_1} + \frac{c}{\lambda_1\nu^2} \normH{w}^4\right) \normH{w(s)}^2
    \Big) ds\\
    &\quad\notag +  c\mu^2 \tau^2 \int_{t_n}^t\norm{\frac{dw(s)}{ds}}_{L^2}^2 ds.
\end{align}

Observing the condition involving $\mu$, $\tau$ and $\kappa$,  $c\mu^2\tau^2 <\tfrac{1}{2}$ holds. Therefore 
   \begin{align}
     \int_{t_n}^{t}\Big\Vert\frac{dw(s)}{ds}\Big\Vert_{L^2}^2\,ds
     \leq
     c\int_{t_n}^{t}\psi(s)\,ds,
   \end{align}
where 
  \begin{align}
    \psi(s) 
    &
    = 
    5\nu^2\normL{\Delta w(s)}^2 
    +
    \left(\frac{cM_1M_2}{\lambda_1^{1/2}} +\frac{cM_0^{2}M_1^{2}}{\nu^2}  +\frac{\mu \nu}{4} + \frac{\mu^2}{\lambda_1} + \frac{c}{\lambda_1\nu^2} \normH{w}^4\right) \normH{w(s)}^2.
  \end{align}

Combining all the above estimates in \cref{H1-estimate}, we obtain for all $t\in[t_n, t_n+\tau)$, 
  \begin{align}
    \frac{d}{dt}\normH{w(t)}^2
    +
    \nu\normL{Aw(t)}^2
    &
    \leq
   \left(-\mu + \frac{c M_1 M_2}{\nu\lambda_1^{1/2}} + \frac{c M_0^2 M_1^2}{\nu^3} + \frac{\mu}{4}\right)\normH{w(t)}^2\chi_{n,\tau}
    \\&\quad\notag
    +
    \frac{c\mu^2\tau}{\nu}\left(\int_{t_n}^{t}\psi(s)\,ds\right)\chi_{n,\tau},
  \end{align}
This inequality is further simplified to 
  \begin{align}\label{ineq1}
    \frac{d}{dt}\normH{w(t)}^2
    +
    \nu\normL{Aw(t)}^2
    &
    \leq
    -\frac{\mu}{2}\normH{w(t)}^2\chi_{n,\tau}
    +
    \frac{c\mu^2\tau}{\nu}\int_{t_n}^{t}\psi(s)\,ds\chi_{n,\tau},
  \end{align}
where we used our condition on $\mu$.

We now argue that 
\begin{align} 
\normH{w(t)}\leq 2M_1 =: R_1\,\,\,\text{for all}\,\,\,t\geq t_0. 
\end{align}
To do so, we first define $\widetilde{t}$ as follows:
\begin{align} 
\widetilde{t} := \sup\{\widetilde{s}\in[t_0, t_0+\tau): \sup\limits_{t\in[t_0, \widetilde{s}]}\normH{w(t)}\leq R_1\}. 
\end{align} 

To show this we will use a bootstrapping argument. In particular we will show $\widetilde{t}\geq t_0+\tau$, and thus a similar argument can be applied to the next time interval $[t_0+\tau, t_1)$. These arguments can then be repeated inductively yielding the desired conclusion.

First, let us assume that $\widetilde{t}\in [t_0,t_0+\tau]$. Integrating \cref{ineq1} in time from $t_0$ to $t\in (t_0, \widetilde{t}\,]$ we obtain:
  \begin{align*}
    &%\notag
    \normH{w(t)}^2 - \normH{w(t_0)}^2 + \nu\int_{t_0}^{t}\normL{\Delta w(s)}^2\,ds
    \leq
    -\frac{\mu}{2}\int_{t_0}^{t}\normH{w(s)}^2
    \\&\quad\notag
    +
    \frac{c\mu^2\tau^2}{\nu}\int_{t_0}^{t}
    \left(
    5\nu^2\normL{\Delta w(s)}^2 
    +
    \left(\frac{cM_1M_2}{\lambda_1^{1/2}} +\frac{cM_0^{2}M_1^{2}}{\nu^2}  +\frac{\mu \nu}{4} + \frac{\mu^2}{\lambda_1} + \frac{cR_1^4}{\lambda_1\nu^2}\right) \normH{w(s)}^2
    \right)\,ds.
  \end{align*}
Collecting similar terms yields the following
\begin{align}
&%\quad \notag
\normH{w(t)}^2 - \normH{w(t_0)}^2 + \nu\left(1 - 5c\mu^2\tau^2 \right) \int_{t_0}^t\norm{\lap w(s)}^2 ds
\\
&\leq \notag
-\left(
\frac{\mu}{2} 
-  \frac{c\mu^2\tau^2M_1M_2}{\lambda_1^{1/2}\nu} 
- \frac{c\mu^2\tau^2M_0^{2}M_1^{2}}{\nu^3}  
- \frac{c\mu^3\tau^2 }{4} 
- \frac{c\mu^4\tau^2}{\lambda_1\nu}
- \frac{c\mu^2\tau^2R_1^4}{\lambda_1\nu^3} 
\right) \int_{t_0}^t\normH{w}^2ds.
\end{align}

We now use our assumptions on the parameters $\mu$ and $\tau$ to obtain the following inequality
\begin{align}
    \normH{w(t)}^2 - \normH{w(t_0)}^2 + \frac{\nu}{4}\int_{t_0}^{t}\normL{\Delta w(s)}^2\,ds
    \leq
    -\frac{\mu}{8}\int_{t_0}^{t}\normH{w(s)}^2\,ds. 
  \end{align}

Note that this now implies in particular that 
\begin{align} 
\int_{t_0}^{t}\normL{\Delta w(s)}^2\,ds \leq \frac{c}{\nu}\normH{w(t_0)}^2. 
\end{align}  
Inserting the above bound into \cref{ineq1}, results in the following: 
  \begin{align}
    &%\quad\notag
    \frac{d}{dt}\normH{w(t)}^2
    \\&\notag
    \leq
    -\frac{\mu}{2}\normH{w(t)}^2
    +
    \frac{c\mu^2\tau}{\nu}\Big( 5\nu^2 + \frac{M_1M_2}{\lambda_1^{3/2}} + \frac{M_0^{2}M_1^{2}}{\nu^2\lambda_1} + \frac{R_1^4}{\nu^2\lambda_1^{2}} + \frac{\mu^2}{\lambda_1^2} + \frac{\mu\nu}{4\lambda_1} \Big)\int_{t_0}^{t}\normL{\Delta w(s)}^2\,ds
    \\&\notag
    \leq
    -\frac{\mu}{2}\normH{w(t)}^2
    +
    c\mu^2\tau\Big( 5\nu + \frac{M_1M_2}{\nu\lambda_1^{3/2}} + \frac{M_0^{2}M_1^{2}}{\nu^3\lambda_1^{1/2}} + \frac{R_1^4}{\nu^3\lambda_1^{3/2}} + \frac{\mu^2}{\nu\lambda_1^2} + \frac{\mu}{4\lambda_1} \Big)\normH{w(t_0)}^2.
  \end{align}
Now integrating the above inequality in time from $t_0$ to $t\in[t_0, \widetilde{t})$ and denoting by 
\begin{align}\label{K-1}
     K_1 := c\left( 5\nu + \frac{M_1M_2}{\nu\lambda_1^{3/2}} 
    + 
    \frac{M_0^{2}M_1^{2}}{\nu^3\lambda_1^{1/2}} 
    +
    \frac{R_1^4}{\nu^3\lambda_1^{3/2}} 
    + 
    \frac{\mu^2}{\nu\lambda_1^2} 
    +
    \frac{\mu}{4\lambda_1}
    \right) 
,
\end{align}
we obtain 
\begin{align}\label{H1_bound_on_w}
    \normH{w(t)}^2
    \leq
    \normH{w(t_0)}^2 e^{-\frac{\mu}{2}(t-t_0)} 
    +
    2\mu\tau K_1 \big(1 - e^{-\frac{\mu}{2}(t-t_0)}\big) \normH{w(t_0)}^2,
\end{align}
which in turn implies
\begin{align}
    \normH{w(t)}^2
    \leq
    R_1^2 e^{-\frac{\mu}{2}(t-t_0)} 
    +
    2\mu\tau K_1 R_1^2\big(1 - e^{-\frac{\mu}{2}(t-t_0)}\big).
\end{align}

We note that with small enough $\tau$ such that $\mu\tau K_1 \leq \frac{1}{4}$ holds we also have $2\mu\tau K_1 R_1^2<R_1^2$.
Thus we find that 
\begin{align} \normH{w(t)}^2 \leq R_1^2, \,\,\,\text{ for all } t\in[t_0, \widetilde{t}\,], \end{align}
holds, and so the bound $\normH{w(t)} \leq R_1$ can be extended beyond $t_0+\tau$. 

It remains to show that the bound $\normH{w(t)} \leq R_1$ is still valid beyond $t_0+\tau$. 
In order to do so, we follow similarly to the argument above; integrating \cref{H1-estimate} in time from $t_0+\tau$ to $t\in[t_0+\tau, t_1)$. As on the interval of integration $\chi_{n,\tau}(t) = 0$, and in view of \cref{ineq1}, we obtain
  \begin{align}
    \frac{d}{dt}\normH{w(t)}^2
    +
    \nu\normL{\Delta w(t)}^2
    &
    \leq
    K_2\normH{w(t)}^2
  \end{align} 
  where 
\begin{align}\label{K-2}
    K_2 
    := 
    \frac{cM_1M_2}{\nu\lambda_1^{1/2}} 
    + 
    \frac{cM_0^2M_1^2}{\nu^3}.
\end{align}
Now dropping the second term on the left side of the above inequality, and using Gr\"onwall's inequality, the \cref{H1_bound_on_w}, as well as the condition $0<\tau=t-t_0\leq\kappa$, we estimate
\begin{align}
    \normH{w(t)}^2
    &
    \leq
    \normH{w(t_0+\tau)}^2\Big(e^{K_2(t - (t_0+\tau))}\Big)
    \\&\notag
    \leq
    \normH{w(t_0)}^2\Big( e^{-\frac{\mu}{2}\tau} + 2\mu\tau K_1(1 - e^{-\frac{\mu}{2}\tau}) \Big)e^{K_2
    (\kappa - \tau)}.
\end{align}

To obtain the decay of $\norm{w(t)}_{H^1}^2$, we now show that the following inequality holds:
\begin{align}\label{theta-H1}
   \theta:= 
   \Big( e^{-\frac{\mu}{2}\tau} + 2\mu\tau K_1(1 - e^{-\frac{\mu}{2}\tau}) \Big)e^{K_2\kappa}  
   < 1.
\end{align}

It suffices to show that the following two inequalities hold simultaneously: 
\begin{align}
\label{H1-small1}
e^{-\frac{\mu}{2}\tau}
\leq
e^{-2K_2\kappa}
\end{align}
as well as 
\begin{align}
\label{H1-small2}
2\mu\kappa  K_1 (1 - e^{-\frac{\mu}{2}\kappa})
\leq
e^{-K_2\kappa} - e^{-2K_2\kappa}.
\end{align}
In view of our condition on the lower bounds of $\mu\tau$, \cref{H1-small1} indeed holds, and note that the last bound on $\mu$ ensures that 
\begin{align}
\frac{\mu}{4K_2}\tau
<
\tau
<
\kappa
\end{align}
remains valid. 
We note that one may choose to increase $\mu$ should one desire a particular value of $\tau$ for fixed $\kappa$.
To enforce \cref{H1-small2}, we estimate the left side as
\begin{align}
2\mu\kappa  K_1 (1 - e^{-\frac{\mu}{2}\kappa})
\leq
2\mu\kappa  K_1 \Big(\frac{\mu}{2}\kappa\Big)
=
\mu^2  K_1\kappa^2
\end{align}
and the right side as
\begin{align}
e^{-K_2\kappa} - e^{-2K_2\kappa}
&
=
e^{-K_2\kappa}\big(1 - e^{-K_2\kappa}\big)
\geq
(1 - K_2\kappa)\Big(K_2\kappa - \frac{K_2^2}{2}\kappa^2\Big).
\end{align}
By the last bound on $\kappa$ in \cref{kappa-bound-H1}, the following quadratic inequality about $\kappa$ holds:
\begin{align}
\frac{K_2^3}{2}\kappa^2 - \Big(\frac{3K_2^2}{2}+\mu^2  K_1\Big)\kappa + K_2
\geq 
0,
\end{align}
which is in fact equivalent to 
\begin{align}
\mu^2  K_1\kappa 
\leq
K_2(1 - K_2\kappa) \Big(1 - \frac{K_2}{2}\kappa\Big).
\end{align}
Thus, \cref{H1-small2} is proved and so it follows that $\theta < 1$.

Inductively, combining the bounds of $\normH{w(t)}$ on $[t_n, t_n+\tau)$ and on $[t_n+\tau, t_{n+1})$, for $n=0,1,2,\cdots$ we obtain 
\begin{align}
\normH{w(t)}\leq \theta^{n}R_1\quad\text{where}\quad0\leq \theta<1, 
\end{align}  
which implies $\normH{v(t) - u(t)}=\normH{w(t)} \to 0$ as $t\to\infty$. 

\end{proof}

\begin{remark}
% We note that our convergence results apply to any interpolant, provided they satisfy \cref{interpolator_time} and the resulting equations are indeed globally well-posed. Interpolants that are implicit with respect to time are challenging to consider analytically as the current state of the feedback-control term depends on the future value of $v$ itself, however it is possible to make sense of them numerically for certain implicit time discretization schemes. 

The above theorem can be extended and applied to the case of interpolants $I_{h,\kappa}$ with a higher-order explicit time-extrapolation component.
% As stated the theorem above applied only to all explicit first order time-extrapolants, but it can be extended to apply for higher-order explicit time-extrapolants as well.
We note that in the above theorem our choice of time-extrapolation is restricted to explicit first order time-extrapolants, as we require a bound on $\int_{t_n}^t \norm{\pd{w}{s}(s)}_{V'}^2 ds$. higher-order time-extrapolants require estimates on higher-order time-derivatives, which can be established inductively. Proof of this statement for higher-order methods is beyond the scope of this work, but will be explored in the sequel.

% While our well-posedness results only extend to first order explicit time-extrapolants, we note that our convergence results extend to implicit interpolants as well, provided they satisfy \cref{interpolator_time} and the resulting equations are indeed globally well-posed.  
\end{remark}

\section{Numerical Results}
% \listoftodos[]
\noindent
In this section we explore how our alterations to the AOT algorithm perform when applied to the 3D incompressible Navier-Stokes equations.

\subsection{Numerical Scheme and Technical Details}

In order to test the AOT algorithm  an ``identical twin'' experimental design was utilized. This is a standard experiment used to test methods of data assimilation, which involves running two separate simulations. This first simulation is used as the reference solution, which is used  to both generate observational data and as the solution we attempt to recover using the AOT algorithm. Next, we start the twin simulation, a second simulation that we begin from a different set of initial data, in this case we initialize the twin simulation to be identically $0$. Using the AOT algorithm we assimilate the observational data from the reference solution into the twin simulation and calculate the error between the two. The phrase identical twin here refers to how both simulations use the same underlying physics, even though the initial data for both simulations is different. For a detailed look at twin experiments see, e.g., \cite{yu2019twin}.

% Here detail the numerical scheme used for 3D NSE on crane
% Cite Taylor code
% Make sure to reference crane and the specific architecture
Simulations of 3D Navier-Stokes equations are performed using a fully dealiased pseudospectral code on the domain $\mathbb{T}^3 = [0,1]^3$. 
Specifically, the derivatives are calculated in Fourier space with products computed in physical space utilizing the the 2/3's dealiasing rule. 
For spatial resolution $N = 512^3$, we use a viscosity of $\nu = 3.58979\times 10^{-4}$, which we find to be in the critical regime for this resolution.
The time stepping has been done using a fully explicit fourth-order Runge-Kutta scheme with the time step restricted in accordance with the advective CFL condition and with the AOT term treated explicitly. Fourier projection onto the first $N$ wavenumbers is utilized as the spatial interpolant for the AOT term, with piecewise constant interpolation used as the time-extrapolant.

For our simulations we first ramp up $u$ from randomized initial state that is rescaled in order to fit a specified energy profile for the spectrum.
% , given by $k^{p}$ for the wavenumber $k$, with $p = \frac{-5}{3}$ for $k = 2$ and $p = \frac{-7}{3}$ for $k > 2$.  
The forcing was Taylor-Green, with components given  explicitly by:
\begin{equation}\label{TG}
\begin{cases}
f_1(x,y,z) &= \sin(2\pi x) \cos(2\pi y)\cos(2\pi z),\\
f_2(x,y,z) &= -\cos(2\pi x) \sin(2\pi y)\cos(2\pi z)\\
f_3(x,y,z) &= 0
\end{cases}
\end{equation}
We simulate the evolution of this equation until time $t = 15$, at which time the spectrum of the solution appears to stabilize. 

In addition to the data assimilation scheme we propose, and the one given in \cite{Foias_Mondaini_Titi_2016}, we also consider the scheme proposed in \cite{Hayden_Olson_Titi_2011} and expanded upon in \cite{Celik_Olson_Titi_2019}, which we will refer to as the Hayden-Olson-Titi (HOT) algorithm. Instead of applying a feedback-control term, the HOT algorithm directly replaces the low Fourier modes of $v$ with those of $u$ at each of the observed times $t_n$.  More explicitly, the scheme is given by:
\begin{equation}
\begin{cases}
v_t &= F(v), \text{ for } t\in (t_{n},t_{n+1})\\
P_N v(t_n) &= P_N u(t_n),
\end{cases}
\end{equation}
where $P_N$ is the projection onto the first $N$ Fourier modes.

We view this heuristically as an example of setting $\tau = 0$, as when $\tau$ decreases to $0$, $\mu$ increases proportionally, and thus the strength of the forcing on the low Fourier modes increases. We heuristically view the direct replacement of the low Fourier modes of $v$ with those of $u$ as a sort of limit of the forcing behavior of the feedback control term.

In order to test these method of data assimilation, we first perform a parameter search in order to determine the optimal choice of truncation number $N$ and $\mu$. To first determine the value of $N$, we use the HOT algorithm to find the smallest value of $N$ for which we obtain convergence within a reasonable amount of simulation time with $\kappa$ fixed. We utilize the HOT algorithm to optimize $N$ by itself as $\mu$ is not present in this scheme. Then, once $N$ has been selected we then use our AOT variant with $\tau = \kappa$ and conduct a parameter search to find an optimal choice of $\mu$. In order to save on computational time and resources for the high resolution simulations, we use the truncation number $N = 100$. We acknowledge that this choice is not optimal; nevertheless, we still choose it in order to reduce computational time for the AOT algorithm. For spatial resolution $512^3$ we find that the choice of $\mu = 5$ is the optimal choice when $\tau = \kappa$ with truncation number $N = 100$.

\subsection*{Results}
Once the truncation number and $\mu$ are chosen for the case of $\tau = \kappa$, we then run simulations for various $\mu$ in the case of $\tau = \frac{\kappa}{10} = \Delta t$. Here $\Delta t = 0.0001$ is given to be our time-step, with $\kappa = 0.001 = 10\cdot \Delta t$. The results our our simulations can be seen in \cref{fig:3D error,fig:vortex tubes}.

In \cref{fig:vortex tubes} we show snapshots of the normalized magnitudes of the vorticities of $u$, $v$, and $u-v$ at various points in time. From this figure one can see that our $v$ evolves quickly from the initial state to mimic the vortex tubes seen in $u$. 
% While the vorticities of both $u$ and $v$ look similar visually at time $t=0.11$ with an error level of approximately $1e-5$.
We note that the vorticities of $u$ and $v$ pictured in this figure are normalized, which is done to provide a consistent coloring for all plots across multiple time-steps. We note that the unscaled version of \cref{fig:vortex tubes} qualitatively shows the same picture, as the difference of the vorticities of $u$ and $v$ is small enough that by time $t=0.11$ it would not be visible without significant rescaling. Moreover, one can see that the difference in vorticity appears to be mostly constant throughout the domain, with slight increases around vortex tubes.

In \cref{fig:3D error} we see the $L^2$ error of our algorithm for varying choices of $\tau$ and $\mu$. In particular, we choose $\tau = \kappa$ with $\mu = 5$ as our baseline for comparison. For this choice of $\tau$, our algorithm reduces to that given in \cite{Foias_Mondaini_Titi_2016}, while $\mu = 5$ is found to be approximately optimal for this particular $\tau$ value. For the remainder of our test we utilize $\tau = \kappa/10 = \Delta t$ and vary our choice of $\mu$.
We found that when $\mu$ is chosen to preserve the impulse of the forcing i.e. $\mu = 5\frac{\kappa}{\tau} = 50$, we obtain similar convergence results to the case for $\mu = 5$ and $\tau = \kappa$. 
However, we also find that when $\tau < \kappa$ we can obtain faster convergence by utilizing a larger $\mu$. Increasing the value of $\mu$ does have diminishing returns. 
We see that as $\mu$ increases the error approaches that of the scheme labeled $\mu \to \infty$, given in \cite{Hayden_Olson_Titi_2011}. We note that while this method does appear to be the most performant, it is not feasible to implement in practice, as it requires particular types of interpolants, such as projection onto low order Fourier modes. In contrast, the AOT algorithm allows for general classes of interpolants.

Additionally, in \cref{fig:3D error} one can the see the performance of our algorithm in the short term. We observe that, as predicted in the analysis, we see exponential decay in the $L^2$ error on all time-steps on which the AOT feedback-control term is applied, that is the time-step immediately after new observational data is measured. This decay is subsequently followed by exponential growth of the error, however with suitably strong $\mu$ the exponential decay dwarfs the growth leading to exponentially fast convergence in time.

\begin{figure}[htb!]
\centering
	\includegraphics[width = \textwidth]{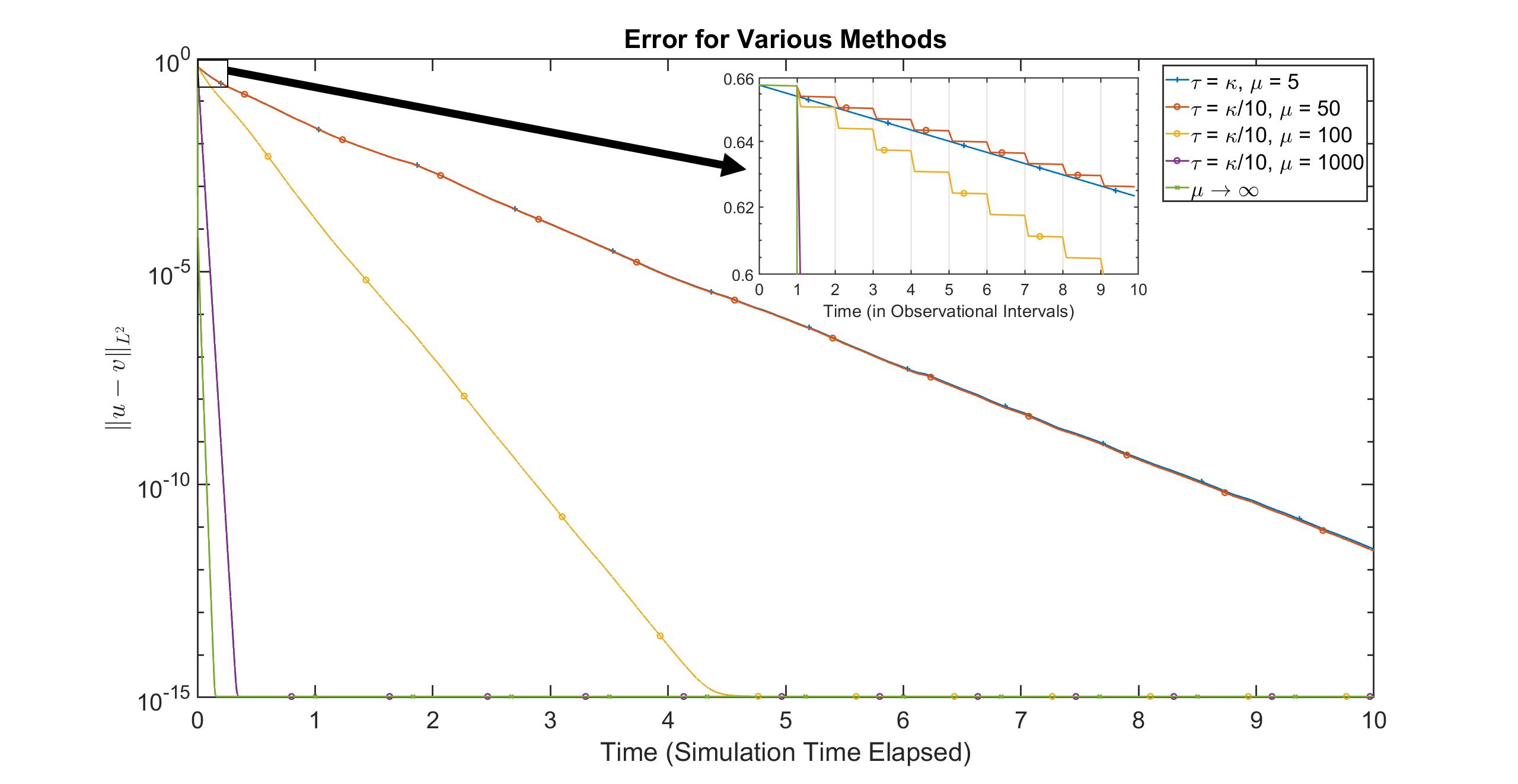}
	\caption{Log-linear plot of $L^2$ error for sparse-in-time observations using different $\tau$ values, with $\kappa = 0.001$ and $\Delta t = 0.0001$ with spatial resolution $512^3$. Inset: A zoom-in of the errors on time interval $[0,0.1]$ (box not drawn to scale) showing the stair-stepping nature of the error decay.}
 \label{fig:3D error}
\end{figure}

\begin{figure}
    \centering    
    \includegraphics[width=.25\textwidth]{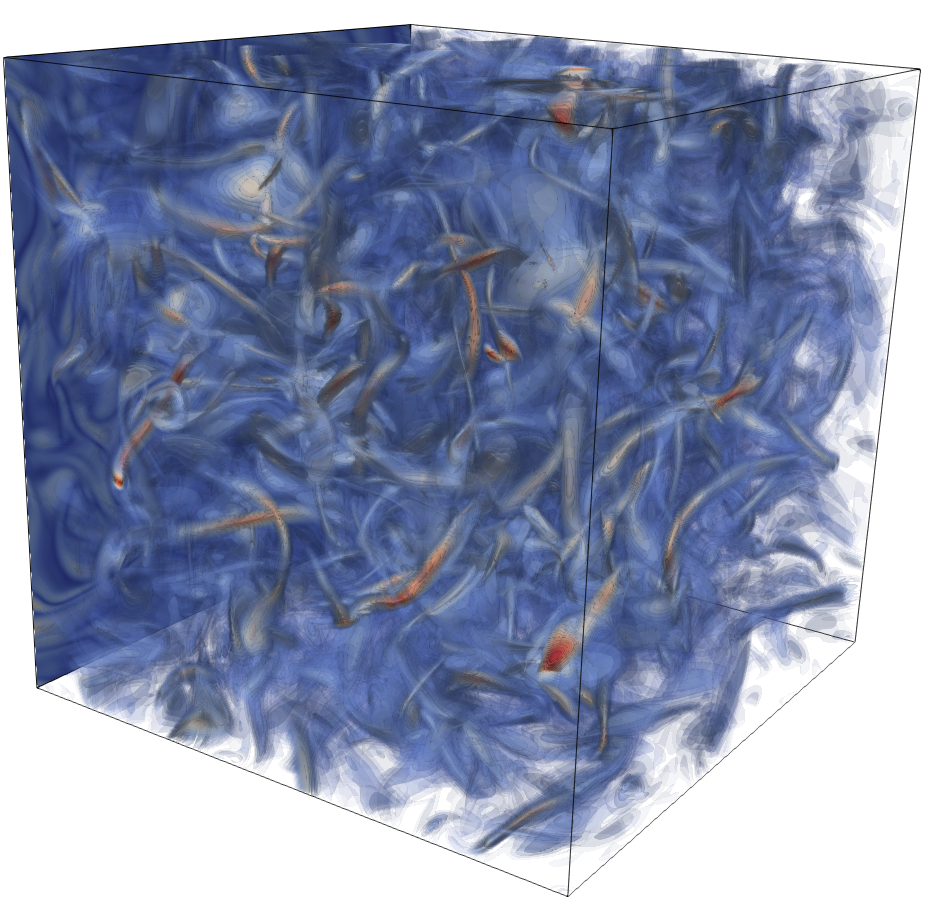}    
    \includegraphics[width=.25\textwidth]{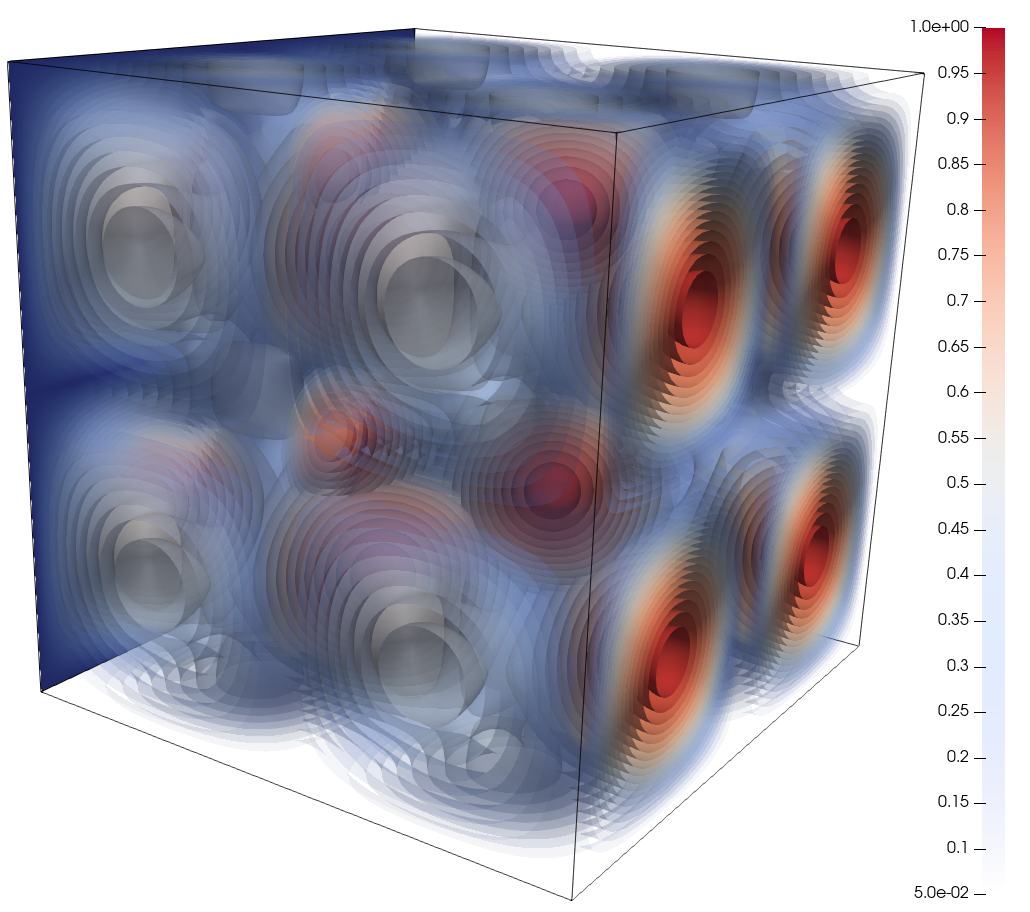}    
    \includegraphics[width=.25\textwidth]{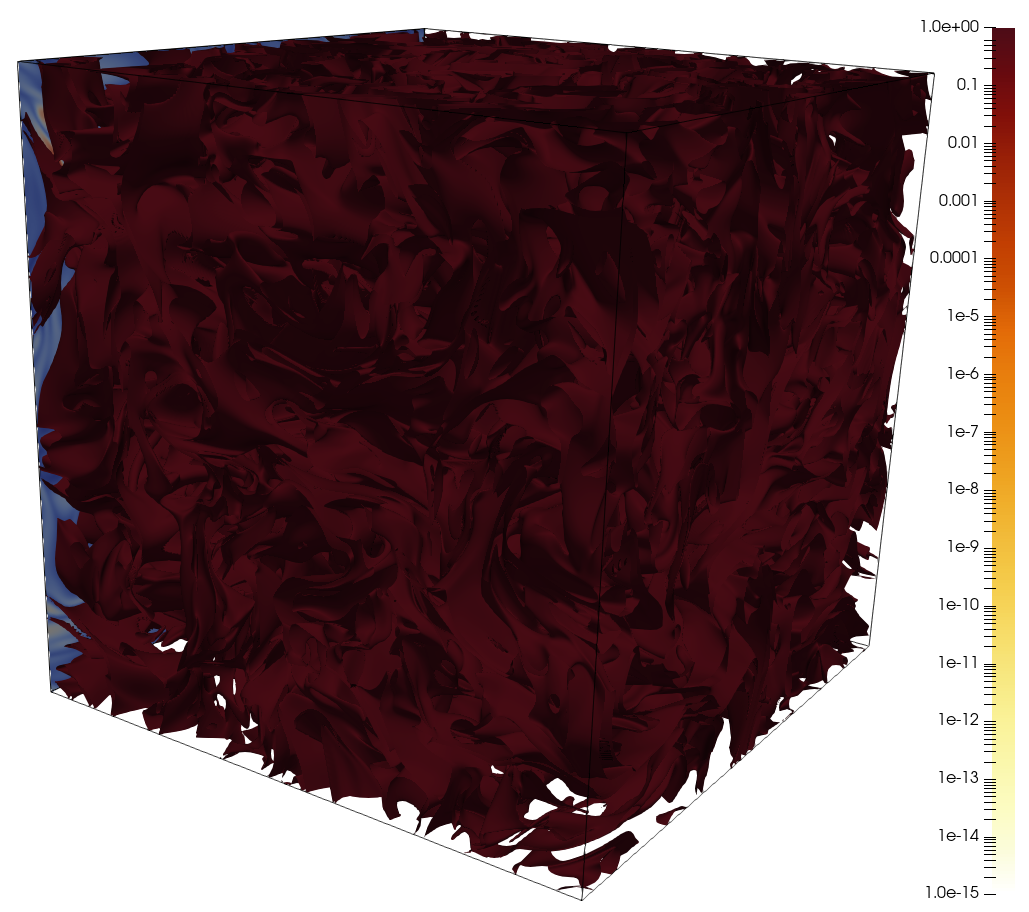}    
    \includegraphics[width=.25\textwidth]{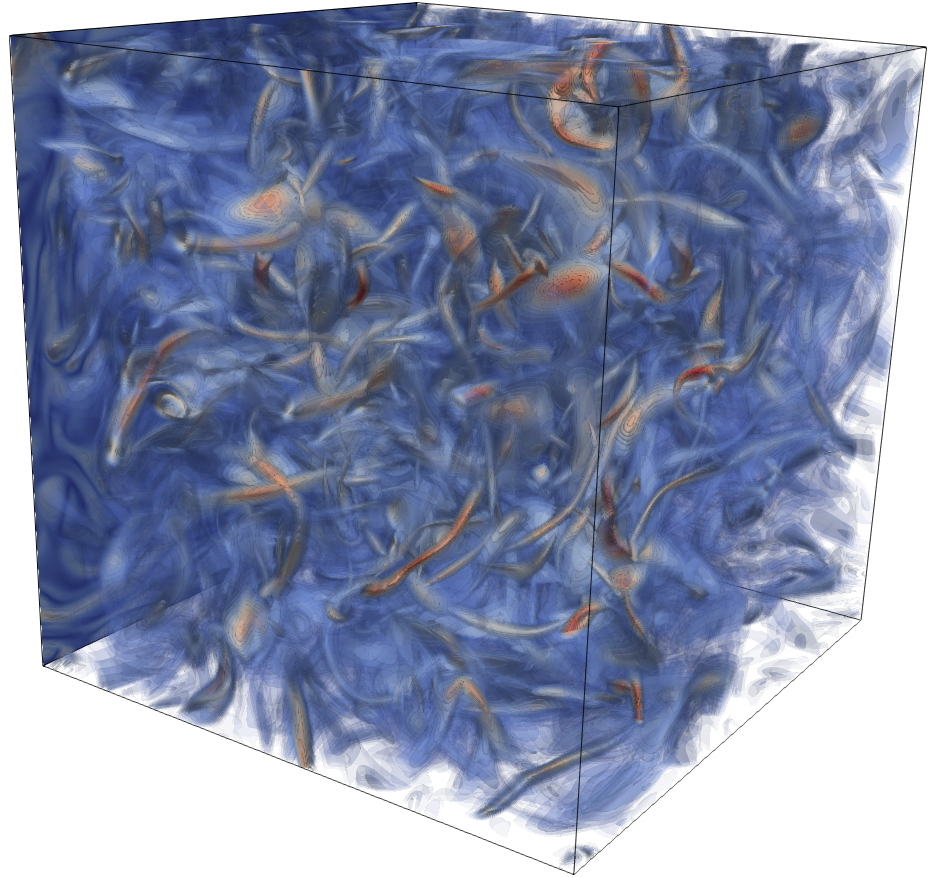}    
    \includegraphics[width=.25\textwidth]{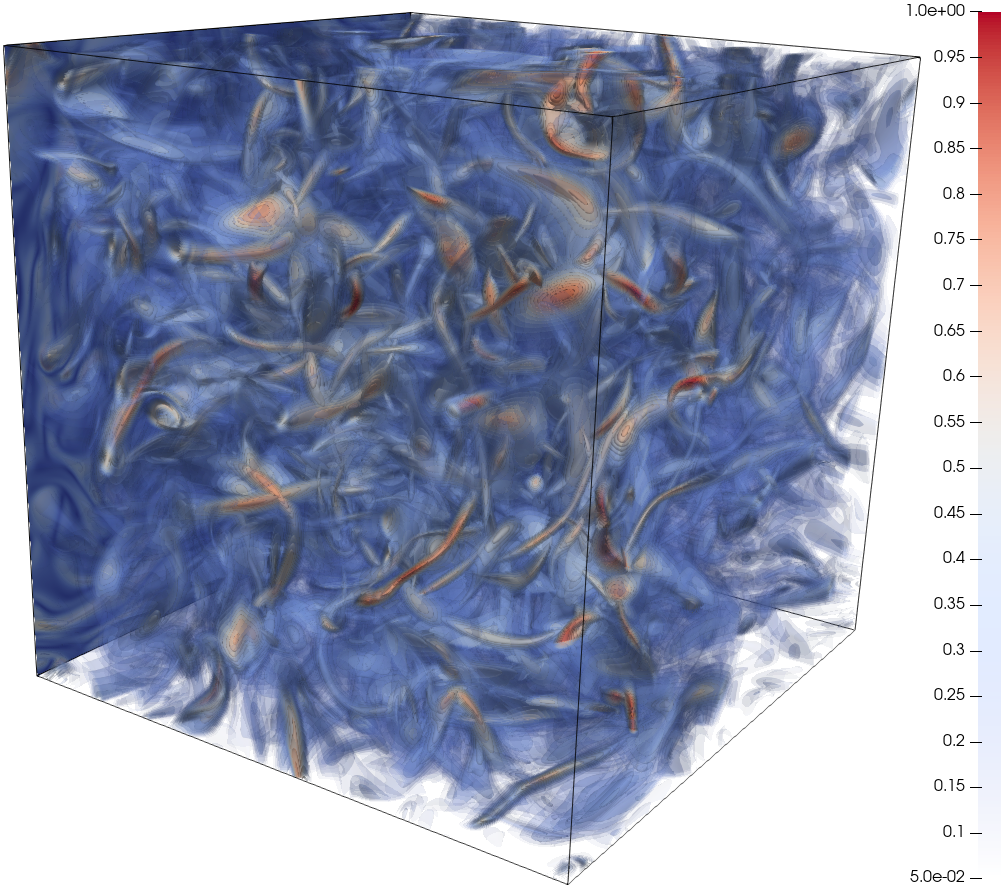}    
    \includegraphics[width=.25\textwidth]{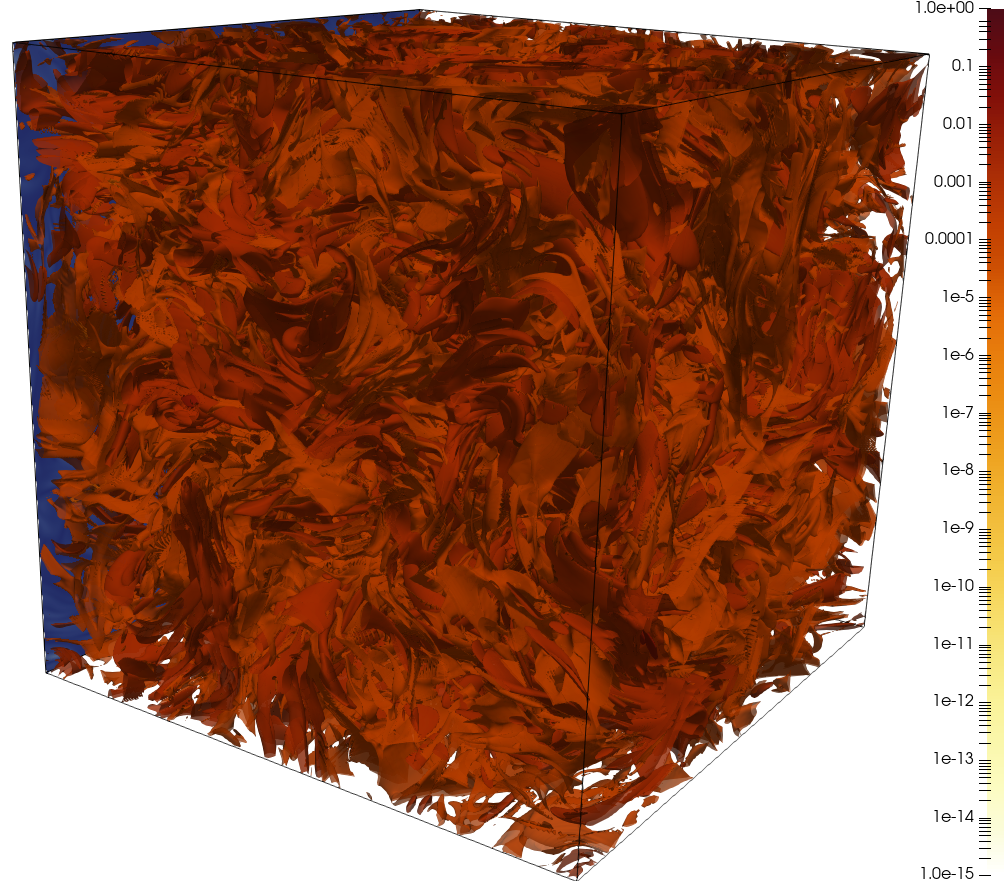}    
    \includegraphics[width=.25\textwidth]{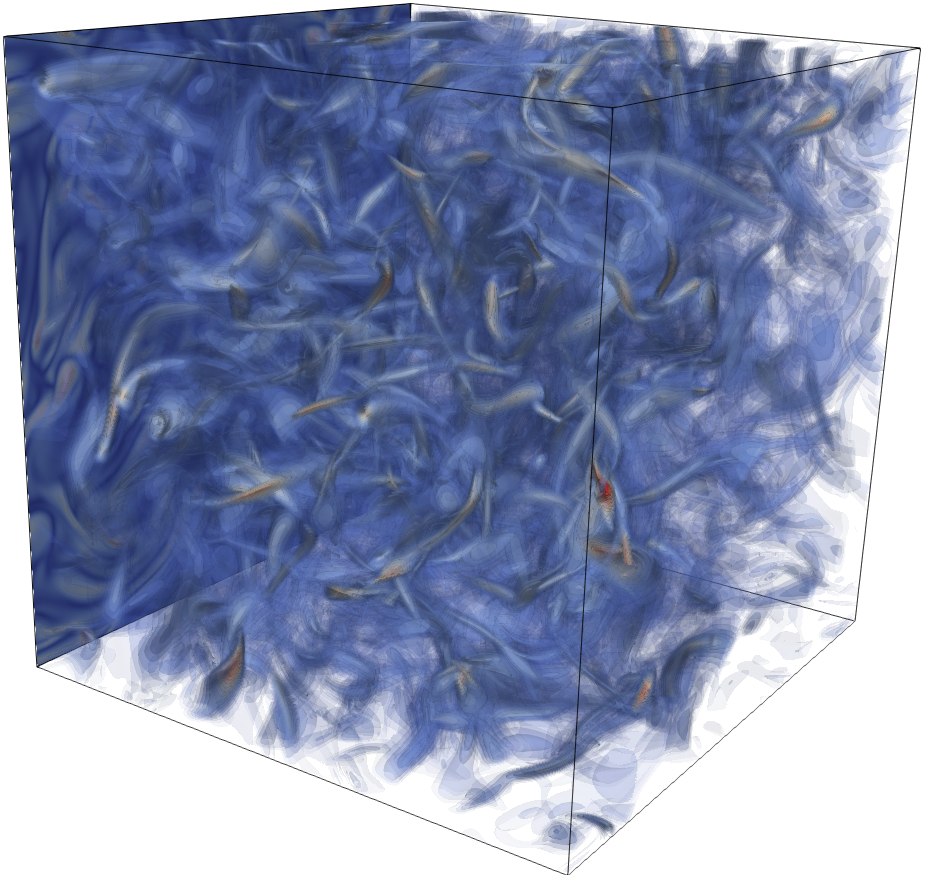}    
    \includegraphics[width=.25\textwidth]{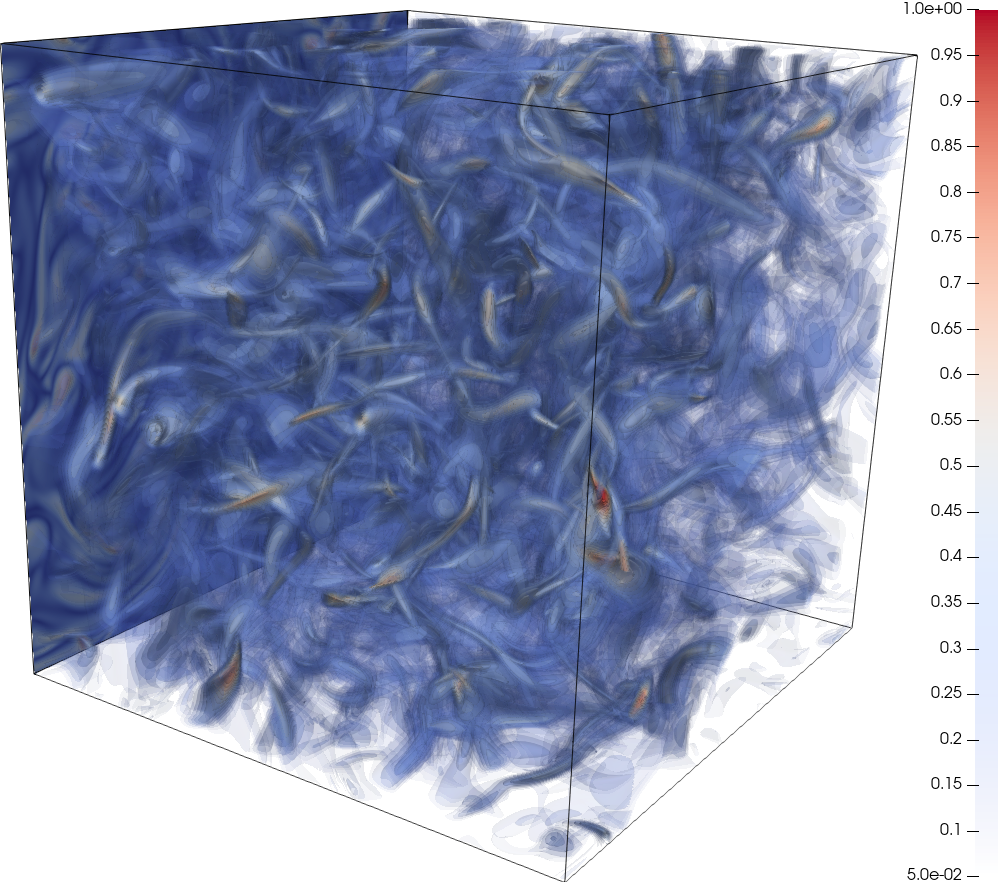}    
    \includegraphics[width=.25\textwidth]{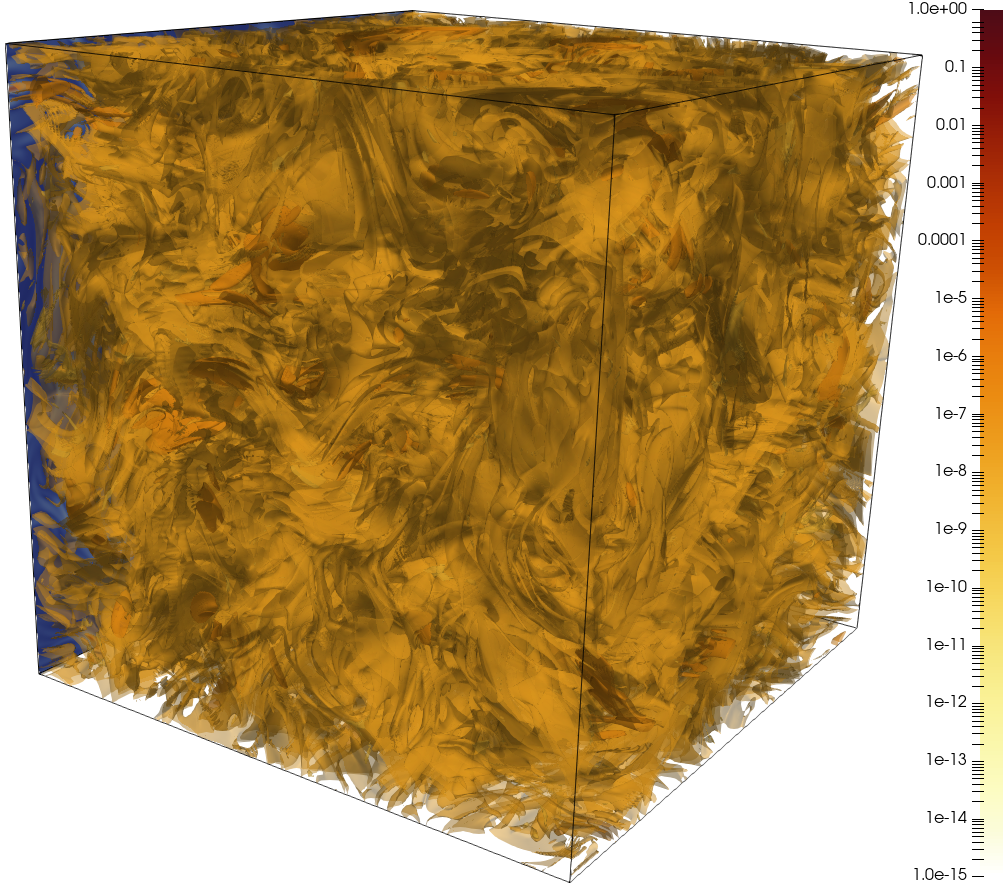}
    \includegraphics[width=.25\textwidth]{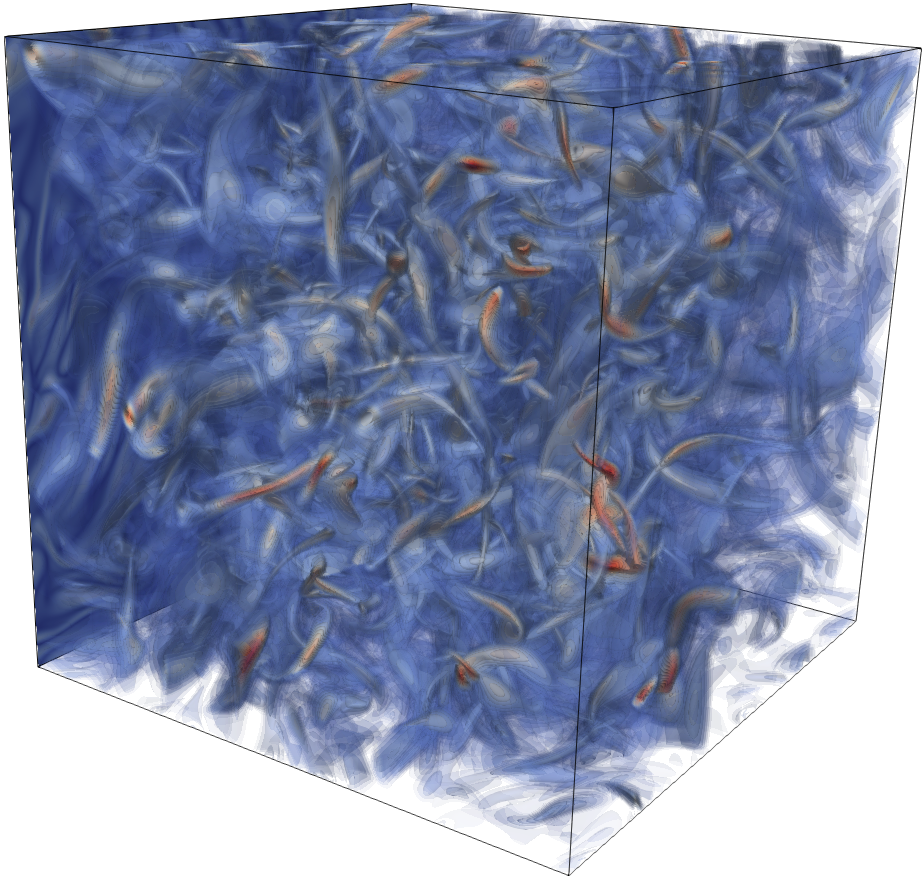}    
    \includegraphics[width=.25\textwidth]{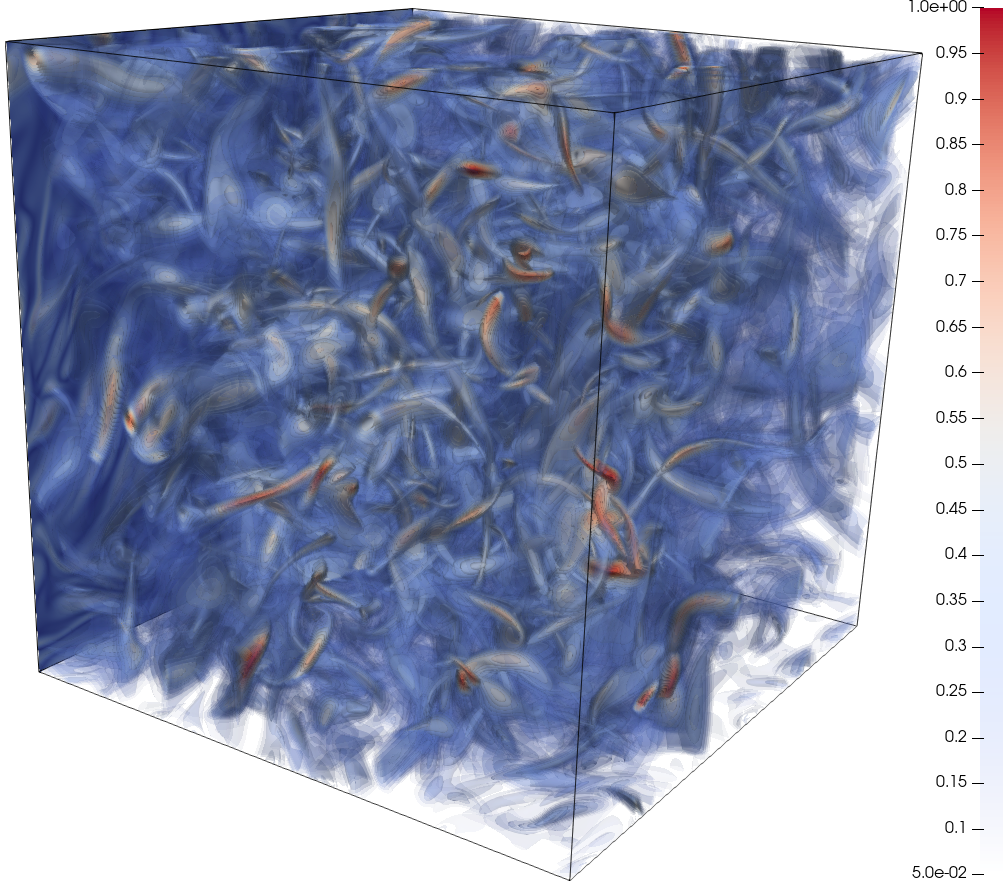}    
    \includegraphics[width=.25\textwidth]{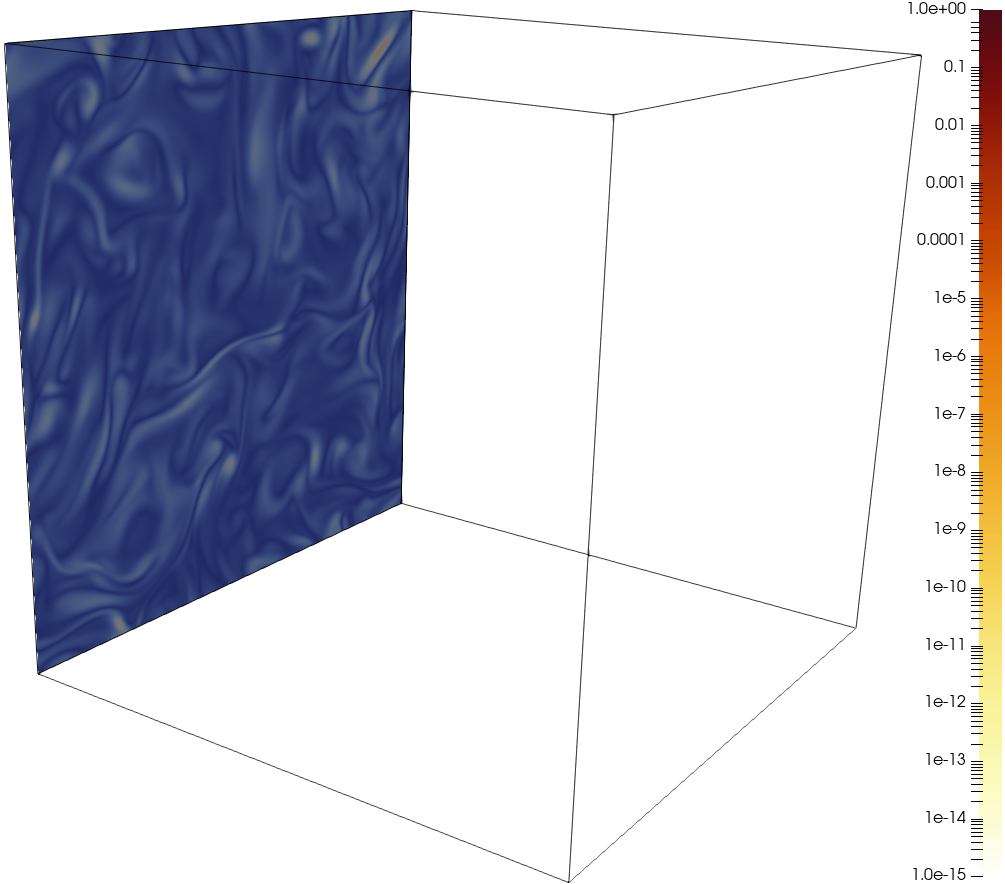}    
    \caption{Vortex tubes of $u$ and $v$ (left and center columns, respectively) shown colored according to magnitude of normalized vorticity with the difference in vorticities (right column) colored based on the magnitude. Each row of this figure is a snapshot of the vorticity magnitude taken at times $t  = 0.001, 0.11, 0.21,$ and $0.31$. For this simulation we used $\tau = \kappa/10$ with $\mu = 1000$.}
    \label{fig:vortex tubes}
\end{figure}

\FloatBarrier

\section{Conclusions}
\noindent
We have seen in this work that not only is it analytically justified\footnote{in the sense that, at least in 2D, global well-posedness holds and exponentially fast convergence to the true solution still occurs} to apply the AOT control over a ``data assimilation window'' which is smaller than the observational time interval, but we have seen computationally that such a scheme can dramatically improve the convergence rate at no additional cost.  The reason for this is straight-forward: by limiting the data assimilation window, one is able to use a significantly larger nudging parameter $\mu$.  The result is that the simulated solution is quickly and forcefully driven toward the observational data, but this forcing is halted swiftly, avoiding instability and allowing the solution time to relax toward the true physics of the system.  In addition, with this fast, furious driving of the solution, the solution is no longer nudged toward a significantly older state of the system. 
We also introduced a new algorithm which allows for linear extrapolation in time of the observed data.

In summary, we find that convergence rates can be greatly accelerated by \textit{(i)} dealing with incoming observations on a very short time interval immediately after they arrive, \textit{(ii)} giving them high priority (large $\mu$) and \textit{(iii)} ceasing to use the observations as soon as possible (to avoid instabilities and to avoid forcing the solution toward data that has become stale).

One may wonder about the limiting case of sending the data assimilation window to zero while sending the nudging strength $\mu\rightarrow\infty$ in some appropriate sense.  In essence, the ``best'' way to do this would be to simply replace the low Fourier modes by the observed Fourier modes (as in \cite{Hayden_Olson_Titi_2011,Celik_Olson_Titi_2019}), but this is not possible in real-world simulations, and hence we view such a replacement strategy as an idealization.  In contrast, while the method we propose here is slightly slower than the idealized replacement method (as we demonstrate in Figure \ref{fig:3D error}, though it is still an order of magnitude faster than previous methods), it is straight-forwardly implementable in a wide variety of real-world applications.

\section*{Acknowledgments}
\noindent
Author A.L. would like to thank the Isaac Newton Institute for Mathematical Sciences, Cambridge, for support and warm hospitality during the programme ``Mathematical aspects of turbulence: where do we stand?'' where work on this manuscript was undertaken. This work was supported by EPSRC grant no. EP/R014604/1. 
Author C.V. would like to thank Los Alamos National Laboratory for kind hospitality while this work was being completed.  
The research of A.L. was supported in part by NSF grants DMS-2206762 and CMMI-1953346.  
The research of C.V. was supported in part by the NSF GRFP grant DMS-1610400.

\bibliographystyle{abbrv}
% \bibliography{LariosBiblio}

\end{document}